\documentclass[11pt]{amsart}
\usepackage{amssymb,amsmath,amsthm,mathrsfs}

\usepackage{a4wide}

\usepackage{enumerate,paralist}

\numberwithin{equation}{section}
\newtheorem{theorem}{Theorem}[section]
\newtheorem{proposition}[theorem]{Proposition}
\newtheorem{lemma}[theorem]{Lemma}
\newtheorem{corollary}[theorem]{Corollary}
\theoremstyle{definition}

\newtheorem{remark}[theorem]{Remark}

\def\SL{\mathrm{SL}(2,\mathbb{C})}
\def\SU{\mathrm{SU}(2)}

\begin{document}

\title[Reidemeister torsion, complex volume, and Zograf infinite product]
{Reidemeister torsion, complex volume, and Zograf infinite product}

\author{Jinsung Park}
\address{School of Mathematics\\ Korea Institute for Advanced Study\\
207-43\\ Hoegiro 85\\ Dong\-daemun-gu\\ Seoul 130-722\\
Korea }
\email{jinsung@kias.re.kr}

\thanks{2010 Mathematics Subject Classification
	57Q10, 32Q45, 58J28.}

\date{\today}

\begin{abstract}
In this paper, we prove an equality which involves Reidemeister torsion, complex volume, and  Zograf infinite product
for closed hyperbolic 3-manifolds.
\end{abstract}

\maketitle


\section{Introduction} \label{s:Introduction}

This paper can be considered as a continuation of a series of papers \cite{P05, P09, GP} of the author concerned on the invariants realized as special values of various dynamical zeta functions for hyperbolic manifolds. In this paper, we restrict our study to the case of closed hyperbolic 3-manifolds and obtain a specific result on a relationship of these invariants.

Although the main machinery to prove the result of this paper is the same as the line of previous works in \cite{P05, P09, GP}  for more general situations,
the motivation of this paper was given from works in a totally different subject, that is, the works of Zograf \cite{Z} and McIntyre-Takhtajan \cite{MT}. Since their works seem to be unrelated to the previous works \cite{P05, P09,GP} directly, it would be helpful for readers to explain how the author has been motivated and suggested by the works in \cite{Z, MT} to consider the problem resolved in this paper.

Let us consider the Teichm\"uller space $\mathfrak{T}_g$ of marked closed Riemann  surfaces of genus $g$.
Each marked closed Riemann surface in $\mathfrak{T}_g$ carries a unique hyperbolic metric.
By \cite{Z, MT}, the following equality holds over $\mathfrak{T}_g$,
\begin{equation}\label{e:ZMT}
\frac{\mathrm{det} \Delta_n}{\mathrm{det} N_n}= c_{g,n} \exp\left(-\frac{6n^2-6n+1}{12\pi}S\right) |F_n|^2.
\end{equation} 
Here $\mathrm{det}\,\Delta_n$ is the regularized determinant of the Laplacian $\Delta_n$ in the hyperbolic metric acting on the space of $n$-differentials, $N_n$ is the Gram matrix of the natural basis of the holomorphic $n$-differentials with respect to the inner product given by the hyperbolic metric, $S$ is the classical Liouville action defined in \cite{TZ}, and $c_{g,n}$ is a constant depending only on $g$ and $n$. The last term $F_n$ is the Zograf infinite product
which defines a harmonic function over $\mathfrak{T}_g$. Let us remark that the log of other two parts
are the K\"ahler potentials of the Weil-Petersson metric on $\mathfrak{T}_g$ respectively. 
The equality \eqref{e:ZMT} can be understood in several different view points. From the number theoretic view point, it can be considered as a higher genus generalization of Kronecker's first limit formula as explained in the introduction of \cite{MT}. It can be also
understood as a \emph{holomorphic factorization} of the determinant of the Laplacian recalling the holomorphic factorization theorem of \cite{BK86}.

All the terms in \eqref{e:ZMT} are defined in terms of structures over a marked closed Riemann surface, so that the equality \eqref{e:ZMT} seems to be unrelated to  higher dimensional hyperbolic manifolds. On the other hand, there are some ways to understand \eqref{e:ZMT} in terms of the corresponding quantities from a hyperbolic $3$-manifold. For a given marked closed Riemann surface, one can consider its Schottky uniformization. This 
defines a hyperbolic $3$-manifold $\mathcal{M}$ called \emph{Schottky hyperbolic} $3$-manifold, which bounds the original Riemann surface 
as a conformal boundary. Then the concerning geometric data defining $F_n$ is given by exactly complex lengths of the closed geodesics in $\mathcal{M}$.  Secondly, by the works of \cite{KS, TT03, PTT}, the classical Liouville action $S$ is related to the renormalized volume $\mathrm{Vol}_r$ of $\mathcal{M}$ by 
\begin{equation}\label{e:S-volume}
S=-4\mathrm{Vol}_r.
\end{equation}
 Here $\mathrm{Vol}_r$ is defined by a regularization process to obtain a finite value for hyperbolic manifold $\mathcal{M}$ of infinite volume.

The above view point in terms of a hyperbolic $3$-manifold suggested the author the existence of the corresponding formula for hyperbolic $3$-manifolds,
in particular, for closed hyperbolic $3$-manifolds. Here another new aspect comes in noticing that the hyperbolic volume always appears with the Chern-Simons invariant to define the following \emph{complex volume}
\begin{equation}\label{e:complex-volume}
\mathbb{V}(\mathcal{M})=\mathrm{Vol}(\mathcal{M}) +i 2\pi^2  \mathrm{CS}(\mathcal{M}) \qquad \text{mod} \quad \mathbb{C}/ i\pi^2 \mathbb{Z}
\end{equation}
for a hyperbolic manifold $\mathcal{M}$ of finite volume. Here $\mathrm{Vol}$ and $\mathrm{CS}$ denote respectively the hyperbolic volume and the Chern-Simons invariant
defined by the Levi-Civita connection for the hyperbolic metric. We refer to \cite{NZ, Y, Zi} for more extensive study on $\mathbb{V}$. 
Moreover, the term involving $F_n$ in \eqref{e:ZMT}  is just given by taking the modulus part for complex valued $F_n$. 
Hence, one can expect that the formula in question might exist as a relationship between complex valued invariants of a hyperbolic $3$-manifold.

The term on the left hand side of \eqref{e:ZMT} is the square inverse of the norm of a canonical section, which is defined by the Quillen metric for a determinant line bundle over $\mathfrak{T}_g$. This also suggests that the corresponding invariant would be related to a metric of the corresponding determinant line bundle for hyperbolic $3$-manifold $\mathcal{M}$.
In this view point, a natural candidate corresponding to the left hand side of
\eqref{e:ZMT} seems to be a complex valued Reidemeister torsion studied by  \cite{BH, BK, CM}. For a closed hyperbolic $3$-manifold $\mathcal{M}$ defined by a Kleinian group $\Gamma$,
let $\rho_{m}$ denote the $m$-th symmetric power of a lifting of the natural holonomy representation from $\Gamma$ to $\mathrm{SL}(2,\mathbb{C})$.
A candidate for the corresponding quantity seems to be 
\begin{equation}
\left(\frac{\mathcal{T}(\mathcal{M}, \rho_m)}{\mathcal{T}_0(\mathcal{M}, \rho_m )} \right)^{-2}.
\end{equation}
Here $\mathcal{T}(\mathcal{M}, \rho_m)$ is the Reidemeister torsion of $\mathcal{M}$ attached to $\rho_m$ and
$\mathcal{T}_0(\mathcal{M},\rho_m)$ is the torsion defined by a complex of spaces of the zero generalized eigensections of non-selfadjoint Hodge
Laplacians.
Indeed  we could prove the  following theorem 

\begin{theorem}\label{t:main-theorem}
For a closed hyperbolic $3$-manifold $\mathcal{M}_\Gamma$ defined by a Kleinian group $\Gamma$,
the following equality holds for $n\geq3$,
\begin{equation}\label{e:main-theorem}
\left(\frac{\mathcal{T}_0(\mathcal{M}_\Gamma, \rho_{2(n-1)})}{\mathcal{T}(\mathcal{M}_\Gamma, \rho_{2(n-1)} )} \right)^{12}=\,  \exp(6\pi i\theta_{2(n-1)}) \exp\left(\frac{2}{\pi} {\left(6n^2-6n+{1}\right)}\mathbb{V}(\mathcal{M}_\Gamma)\right) F_n^{\, 12}.
\end{equation}
Here $\theta_{2(n-1)}\in\mathbb{R}$ is given by a linear combination of eta invariants.  
\end{theorem} 

Precise definitions of terms in \eqref{e:main-theorem} are given at \eqref{e:def-zero-torsion}, \eqref{e:def-Reide}, \eqref{e:cn-eta}, \eqref{e:def-complex-vol}, and \eqref{e:def-F-final} respectively.
The term $F_n$ in \cite{MT} contains an additional normalization factor.

\begin{remark}\label{r:main-rmk} 
One can see that the formula \eqref{e:main-theorem} is exactly compatible  with \eqref{e:ZMT} taking modulus of both sides
of \eqref{e:main-theorem} by \eqref{e:S-volume} and \eqref{e:complex-volume}.  The extra power by $6$ in \eqref{e:main-theorem} is due to ambiguity
of the imaginary part of $\mathbb{V}$ up to $i\pi^2\mathbb{Z}$. In order to define a lifting of the holonomy representation of $\Gamma$ to $\mathrm{SL}(2,\mathbb{C})$, we need to use a spin structure of $\mathcal{M}_\Gamma$, but  an even symmetric power representation $\rho_{2(n-1)}$ is independent of this choice of the spin structure. Hence, the equality \eqref{e:main-theorem} does not depend on any 
choice of a spin structure of $\mathcal{M}_\Gamma$.
A similar equality for $\rho_{2n-1}$ with $n\geq 3$ also holds as given in \eqref{e:main-odd-G} where a choice of spin structure is involved.
\end{remark}

It seems to be interesting problem to extend \eqref{e:main-theorem} to a noncompact hyperbolic $3$-manifold. For a hyperbolic $3$-manifold with cusps, 
combining the works in \cite{P05, P09, GP, Zi} one can try to obtain a generalization of \eqref{e:main-theorem}  possibly with a new contribution from cusps.
Such a generalization would be related to the study of \emph{volume conjecture}  \cite{K,MM01} where the hyperbolic $3$-manifold with a cusp is given by the knot complement  from $S^3$.
Note that  the complex volume and the Reidemeister torsion are
two leading parts of a certain expansion derived from the $N$-colored Jones polynomial as $N\to\infty$. 
Another possible direction of generalization of \eqref{e:main-theorem} is the case of a convex cocompact hyperbolic $3$-manifold with infinite volume.
As mentioned before, this is closely related to the study of renormalized volume \cite{KS, TT03, PTT} for this case. A related formula with Bergman tau function is proved in \cite{MP} for a Schottky hyperbolic $3$-manifold.

Here is a structure of this paper. In the section 2, we review some basic materials which are needed in the next sections. In the section 3, we introduce various zeta functions and prove a primitive version of Theorem \ref{t:main-theorem}.
In the section 4,  we prove formulae between invariants derived from the Selberg trace formulae and review the work of \cite{CM}. In the section 5, we prove the main theorem combining all the results proved in the previous sections.

\subsection*{Acknowledgements}
The work of the author was partially supported by SRC - Center for Geometry and its Applications - grant No.~2011-0030044.
He thanks R.~Wentworth for useful remarks which were helpful to improve the result of this paper.

\section{Basic Materials}

\subsection{Hyperbolic 3-space as a symmetric space}
Let $G=\SL$ and $K=\SU$ be a maximal compact subgroup of $G$. Recall that $G$ is a double cover of $\mathrm{PSL}(2,\mathbb{C})$ which 
is the isometry group of the hyperbolic 3-space $\mathbb{H}^3$. The action of $\mathrm{PSL}(2,\mathbb{C})$ is given by
\begin{equation}\label{e:gamma-action}
\begin{pmatrix} a & b\\ c & d \end{pmatrix} q = \frac{aq +b}{cq+d}  
\end{equation}
where $\left(\small{\begin{smallmatrix} a & b\\ c & d \end{smallmatrix}}\right)\in \mathrm{PSL}(2,\mathbb{C})$ and $q=z+tj$ is the quoternion representation of a point $(z,t)\in \mathbb{H}^3$.
Therefore $G$ also acts on $\mathbb{H}^3$ and the isotropy subgroup of $(0,1)\in \mathbb{H}^3$ is $\SU$, hence $\mathbb{H}^3 \cong G/K$.
From now on, we use this realization of the hyperbolic 3-space as a symmetric space to apply some basic harmonic analysis over a symmetric space.

Let $G=NAK$ be the Iwasawa decomposition of $G$  where
\begin{equation}
N=\left\{ \begin{pmatrix} 1 & x+iy \\ 0 & 1 \end{pmatrix} \ | \ x,y\in \mathbb{R}  \right\}, \qquad A=\left\{ \begin{pmatrix} e^u & 0 \\ 0 &e^{-u} \end{pmatrix} \ | \ u \in\mathbb{R} \right\}.
\end{equation}
We assume that the Haar measures of  $N,A,K$ are given by $dn=dxdy, du, dk$ respectively where $dk$ has the total mass $1$.  
A Cartan subgroup $T$ of $G$
is given by $AM$  where
\begin{equation}
M=\left\{ \begin{pmatrix} e^{i\theta} & 0 \\ 0& e^{-i\theta} \end{pmatrix}
 \  | \  \theta\in [0,2\pi] \right\}.
\end{equation}
We take the Haar measure of $T$ to be $ dt=\frac{1}{2\pi} du d\theta$. The set of unitary characters of $M$ denoted by $\hat{M}$ is parametrized by 
$k\in\mathbb{Z}$. The character corresponding to $k$ is given by
\begin{equation}
\sigma_k \biggl( \begin{pmatrix} e^{i\theta} & 0 \\ 0& e^{-i\theta} \end{pmatrix} \biggr) = e^{ik\theta}.
\end{equation}

Let $\frak{g}$, $\frak{k}$, $\frak{n}$, $\frak{a}$, and $\frak{m}$ be the Lie algebras of $G,K,N, A$, and $M$ respectively. Let 
\begin{equation}
\frak{g}=\frak{k}\oplus \frak{p}
\end{equation}
be the Cartan decomposition of $\frak{g}$ given by the Cartan involution $\theta$ where $\mathfrak{k}$, $\mathfrak{p}$ are the $1,-1$ eigenspaces of 
$\theta$ respectively. Let $\alpha$ be the unique positive root of $(\frak{g},\frak{a})$. Let $H\in\frak{a}$ be such that
$\alpha(H)=1$. Let $\frak{a}^+\subset \frak{a}$ be the positive Weyl chamber and $A^+=\exp(\frak{a}^+)$. Put $\frak{h}=\frak{a}\oplus \frak{m}$. Then $\frak{h}$ is a Cartan subalgebra of $\frak{g}$.

The Cartan-Killing form $C$ is positive definite on $\frak{p}$ and is negative definite on $\frak{k}$. We may identify $\frak{p}$ with the tangent space
to $G/K$ at the identity coset. Then $C$ provides with an invariant metric on $G/K$. We use a normalized symmetric bilinear form defined by
\begin{equation}
C_0( X, Y ) =\frac14 C(X, Y ) \qquad \text{for \quad $X,Y\in\frak{g}$},
\end{equation}
so that the corresponding invariant metric has the constant curvature $-1$. Let $\{Z_i\}$ be an orthonormal basis for $\frak{k}$ with respect to $-C_0$ and
$\{Z_j\}$ be an orthonormal basis for $\frak{p}$ with respect to $C_0$. Then the normalized Casimir elements $\Omega$, $\Omega_K$ in the universal enveloping algebra of $\frak{g}_{\mathbb{C}}$, $\frak{k}_{\mathbb{C}}$ are given by
\begin{equation}
\Omega= -\sum_{i} Z_i^2 +\sum_j Z_j^2, \qquad \Omega_K= -\sum_{i} Z_i^2.
\end{equation}
For the right regular representation $R$ of $G$ on $C^{\infty}(G)$ defined by $R(g_2) f(g_1) = f(g_1g_2)$ for $g_1,g_2 \in G$ and $f\in C^\infty(G)$,
the normalized Casimir element induces a differential operator denoted by $R(\Omega)$.


Let $\Gamma$ be a cocompact torsion free discrete subgroup of $G$. Then $\mathcal{M}_\Gamma:=\Gamma\backslash \mathbb{H}^3$ is a compact hyperbolic manifold by definition and this is a special case of a locally symmetric space with a realization of the double coset space $\Gamma\backslash G/K$. Any compact hyperbolic 3-manifold has such a realization. 
Since we assume that $\Gamma$ is a discrete subgroup of $G=\mathrm{SL}(2,\mathbb{C})$ rather than 
$\mathrm{PSL}(2,\mathbb{C})$, the resulting manifold $\mathcal{M}_\Gamma$ is equipped with a spin structure.

For a nontrivial $\gamma\in \Gamma$, there exist $g\in G$,  $a_\gamma \in A^+$, and $m_\gamma\in M$ such that
\begin{equation}\label{e:gamma-conj}
g\gamma g^{-1}= a_\gamma m_\gamma.
\end{equation} 
It is known that 
$a_\gamma$ depends only on $\gamma$ and $m_\gamma$ is determined by $\gamma$ up to conjugacy in $M$ (see Lemma 6.6 of \cite{Wall}).
By definition, there exist $\ell_\gamma >0$ and $\theta_\gamma\in [0,2\pi]$ such that
\begin{equation}\label{e:hyp-a-m}
a_\gamma= \exp(\ell_\gamma H)=\begin{pmatrix} e^{\ell_\gamma/2} & 0 \\ 0 & e^{-\ell_\gamma/2} \end{pmatrix}, 
\qquad m_\gamma= \begin{pmatrix} e^{i\theta_\gamma/2} & 0 \\ 0 &  e^{-i\theta_\gamma/2} \end{pmatrix}.
\end{equation}
From \eqref{e:gamma-action}, it follows that $a_\gamma m_\gamma$ acts on $\mathbb{H}^3$ by $(z,t)\to (e^{\ell_\gamma+i\theta_\gamma}z, e^{\ell_\gamma}t)$.
The positive real number $\ell_\gamma$ is the length of the unique closed geodesic $C_\gamma$ in $\mathcal{M}_\Gamma$ that corresponds to the conjugacy class of $\gamma$ in $\Gamma$.  A closed geodesic $C_\gamma$ also  corresponds to a fixed point of the geodesic flow on the unit sphere bundle $\Gamma\backslash G/M$ over $\mathcal{M}_\Gamma\cong \Gamma\backslash G/K$. Its tangent bundle is given by $\Gamma\backslash G\times_M (\bar{\mathfrak{n}}\oplus \mathfrak{a}\oplus \mathfrak{n})$ where $\bar{\mathfrak{n}}=\theta(\mathfrak{n})$ and $M$ acts on $\bar{\mathfrak{n}}\oplus
\mathfrak{a}\oplus\mathfrak{n}$ by the adjoint action $\mathrm{Ad}$. The Poincar\'e map $P(C_\gamma)$ is the differential of the geodesic flow
at $C_\gamma$, which is given by $P(C_\gamma)=\mathrm{Ad}(a_\gamma m_\gamma)$ if $\gamma=a_\gamma m_\gamma$. 
Now we put
\begin{equation}\label{e:def-D-gamma}
\begin{split}
D(\gamma):=& \Big|  \mathrm{det} \left( \mathrm{Ad}(a_\gamma m_\gamma)\big|_{\bar{\mathfrak{n}}\oplus \mathfrak{n}}  - \mathrm{Id}  \right)   \Big|^{1/2}
=   e^{-\ell_\gamma} \Big|\mathrm{det} \left( \mathrm{Ad}(a_\gamma m_\gamma)\big|_{{\mathfrak{n}}}-\mathrm{Id} \right)\Big| \\
=& e^{\ell_\gamma} \mathrm{det} \left( \mathrm{Id}- \mathrm{Ad}(a_\gamma m_\gamma)\big|_{\bar{\mathfrak{n}}} \right)
=e^{\ell_\gamma} (1-e^{-(\ell_\gamma+i\theta_\gamma)}) (  1-e^{-(\ell_\gamma-i\theta_\gamma)}).
\end{split}
\end{equation}

A nontrivial $\gamma\in \Gamma$ is called \emph{primitive} if it can not be written as $\gamma=\gamma_0^k$ for some other $\gamma_0\in \Gamma$ and $k>0$.
For any nontrivial $\gamma\in \Gamma$, there exists a unique primitive element $\gamma_0\in\Gamma $  and $n_\gamma\in \mathbb{N}$ such that
$\gamma=\gamma_0^{n_\gamma}$.

\subsection{Bundles induced by representations}\label{ss:induced bundle}
By Proposition 2.2.3 of \cite{GW}, for an integer $m\geq 0$, there exists a unique (up to equivalence) irreducible representation 
\begin{equation}\label{e:def-rho}
\rho_m : G \to \mathrm{GL} (S^m(\mathbb{C}^2)),
\end{equation}
which is given by the $m$-th symmetric power of the standard representation of $G$ on $\mathbb{C}^2$. The restrictions of $\rho_m$ to $AM$ decomposes as follows:
\begin{equation}\label{e:decomp-rho}
\rho_m|_{AM} = {\bigoplus}^m_{k=0} \ e^{(\frac{m}{2}-k)\alpha} \otimes \sigma_{m-2k}.
\end{equation}

For a finite dimensional irreducible representation $(\chi, V_\chi)$ of $\Gamma$, we define a flat vector  bundle $E_\chi$ over $\mathcal{M}_\Gamma=\Gamma\backslash G/K$ by
\begin{equation}\label{e:def-flat-bundle}
E_\chi = \Gamma\backslash ( G/K \times V_\chi)
\end{equation}
where $\Gamma$ acts on $G/K\times V_\chi$  by $\gamma (gK, v)= (\gamma gK, \chi(\gamma)v)$. 
In this paper, we mainly use the restriction of the representation $\rho_m$ in \eqref{e:def-rho} to $\Gamma$ to define a flat vector bundle by
\eqref{e:def-flat-bundle}. Throughout this paper, we denote by $E_{\rho_m}$ the resulting flat vector bundle over $\mathcal{M}_\Gamma$.

For a finite dimensional irreducible representation $(\tau, V_\tau)$ of $K$, we also define a locally homogeneous vector bundle $E_\tau$ over $\mathcal{M}_\Gamma$ by
\begin{equation}\label{e:def-E-tau}
E_\tau= (\Gamma\backslash G \times V_\tau)/ K
\end{equation}
where $K$ acts on $\Gamma\backslash G\times V_\tau$ by $(\Gamma g,v) k = (\Gamma gk, \tau(k)^{-1} v)$.

For $m\geq 0$, we denote by $\tau_m$ the irreducible representation of $K$ given by the restriction of $\rho_m$ to $K$. By \eqref{e:decomp-rho},
\begin{equation}\label{e:decomp-tau}
\tau_m|_{M} = \bigoplus^m_{k=0} \sigma_{m-2k}.
\end{equation}
Let $R(K)$ and $R(M)$ denote the representation rings of $K$ and $M$ respectively. The inclusion $\imath:M \to K$ induces the restriction map $\imath^*: R(K)\to R(M)$.
By \eqref{e:decomp-tau}, 
\begin{equation}\label{e:ind-rest-map}
\begin{split}
&\imath^*(\tau_m-\tau_{m-2}) =\sigma_m+\sigma_{-m}, \qquad \text{for}\quad m\geq 2,\\
&\imath^*(\tau_1)= \sigma_1+\sigma_{-1}, \qquad  \imath^*(\tau_0) =\sigma_0.
\end{split}
\end{equation}
Note that $\tau_1$ is the \emph{spin representation} of $K$ and $\sigma_1,\sigma_{-1}$ are the \emph{half spin representations} of $M$. 

\subsection{Eta invariant and Chern-Simons invariant}\label{ss:eta}

The locally homogeneous vector bundle $E_{\tau_1}$ defined by the spin representation $\tau_1$ is equipped with the Dirac operator 
\begin{equation}\label{e:def-Dirac}
{D} s = \sum_{i=1}^3 c(e_i) \nabla_{e_i}^{\tau_1} s \qquad \text{for} \quad s\in C^\infty(\mathcal{M}_\Gamma, E_{\tau_1}) 
\end{equation}
where $c(e_i)$ denotes the Clifford multiplication of an local orthonormal frame $\{e_i\}$ of $T\mathcal{M}_\Gamma$ and $\nabla^{\tau_1}$ denotes the unique locally $G$-invariant connection of $E_{\tau_1}$.
Then canonically the vector bundle $E_{\tau_1}\otimes E_{\tau_{k-1}}$ is equipped with the Dirac operator defined in a similar way to \eqref{e:def-Dirac}
replacing $\nabla^{\tau_1}$ by $\nabla^{\tau_1} \otimes \mathrm{Id}+ \mathrm{Id}\otimes \nabla^{\tau_{k-1}}$ where $\nabla^{\tau_{k-1}}$ denotes the unique 
locally $G$-invariant
connection of $E_{\tau_{k-1}}$.
We denote by $D(\sigma_k)$ the resulting Dirac operator acting on $C^\infty(\mathcal{M}_\Gamma, E_{\tau_1\otimes \tau_{k-1}})$.  
The representations $\sigma_k$ and $\tau_{k-1}$ are related by
\begin{equation}
\sigma_k-\sigma_{-k} = (\sigma_1-\sigma_{-1}) \otimes i^*(\tau_{k-1})
\end{equation}
where $i^*:R(K)\to R(M)$. By the above construction, $D(\sigma_1)$ denotes the Dirac operator defined by a spin structure,
and $D(\sigma_2)$ denotes the odd signature operator. The Dirac operator $D(\sigma_k)$ is a first order selfadjoint differential operator with
spectrum consisting of real eigenvalues of finite multiplicities $\{\lambda_\ell\}_{\ell\in\mathbb{Z}}$.
The eta function $\eta(D(\sigma_k),s)$ is defined by
\begin{equation}
\eta(D(\sigma_k),s)= \sum_{\lambda_\ell>0} \lambda_\ell^{-s}  - \sum_{\lambda_\ell< 0} (-\lambda_\ell)^{-s} \qquad \text{for \quad $\mathrm{Re}(s)\gg 0$},
\end{equation}
which has a meromorphic extension to $\mathbb{C}$ and is regular at $s=0$.
The eta invariant of $D(\sigma_k)$ is defined by
\begin{equation}\label{e:def-eta}
\eta(D(\sigma_k))=\eta(D(\sigma_k),0).
\end{equation}
We refer to \cite{APSI, Gil}  for more details on the eta invariant.

As explained in the section 3 of \cite{Y},
the following $3$-form is  a left-invariant on $\mathrm{PSL}(2,\mathbb{C})$, which can be identified with the frame bundle $F(\mathbb{H}^3)$,
\begin{equation}
\begin{split} 
C=&\frac{1}{4\pi^2} \big( 4\theta_1\wedge \theta_2\wedge \theta_3 -d(\theta_1\wedge\theta_{23} +
\theta_2\wedge\theta_{31} +\theta_3\wedge\theta_{12}) \big)\\
  &\ \ +\frac{i}{4\pi^2} \big( \theta_{12}\wedge\theta_{13}\wedge \theta_{23} -\theta_{12}\wedge\theta_1\wedge\theta_2-\theta_{13}\wedge\theta_{1}\wedge\theta_3 -\theta_{23}\wedge\theta_2\wedge\theta_3 \big).
\end{split}
\end{equation} 
Here $\theta_i$, $\theta_{ij}$ denote the fundamental form and the connection form respectively of the Levi-Civita connection on $F(\mathbb{H}^3)$.
Since it is left-invariant, it descends to the frame bundle $F(\mathcal{M}_\Gamma)\cong \Gamma\backslash F(\mathbb{H}^3)$ where $\Gamma\subset \mathrm{SL}(2,\mathbb{C})$ acts on $F(\mathbb{H}^3)$ in an obvious way.
Now the complex volume of $\mathcal{M}_\Gamma$ is defined by
\begin{equation}\label{e:def-complex-vol} 
\frac{1}{\pi^2} \mathbb{V}(\mathcal{M}_\Gamma):=\frac{1}{\pi^2}\left(\mathrm{Vol}(\mathcal{M}_\Gamma)+i2\pi^2\mathrm{CS}(\mathcal{M}_\Gamma)\right) =
\int_{\mathcal{M}_\Gamma} s^*C \qquad \mathrm{mod} \quad i \mathbb{Z}.
\end{equation}
Here $s$ denotes a section from $\mathcal{M}_\Gamma$ to $F(\mathcal{M}_\Gamma)$ and
the ambiguity in the phase part of \eqref{e:def-complex-vol} by $i\mathbb{Z}$ is due to a choice of $s$. The Chern-Simons invariant has the following
relation with the eta invariant of the odd signature operator  $D(\sigma_2)$ over $\mathcal{M}_\Gamma$,
\begin{equation}\label{e:APS}
2\, CS(\mathcal{M}_\Gamma)= 3\, \eta(D(\sigma_2)) \qquad \mathrm{mod} \ \ \mathbb{Z}.
\end{equation}
Actually this equality holds for any closed Riemannian 3-manifold $\mathcal{M}$. We refer to \cite{APS} for more details about this formula.

\subsection{Hodge Laplacian}\label{ss:hodge} We begin with a general case of a Riemannian manifold.
Let $\mathcal{M}$ be an oriented  Riemannian manifold of dimension $n$. For the differential
$d:\Omega^{p-1}(\mathcal{M}) \to \Omega^p(\mathcal{M})$, its formal adjoint operator $d^*:\, \Omega^p(\mathcal{M}) \to \Omega^{p-1}(\mathcal{M})$ is defined by
\begin{equation}\label{e:adjoint-d}
d^*= (-1)^{np +n+1} \star d \star 
\end{equation} 
where $\star: \Omega^p(\mathcal{M}) \to \Omega^{n-p}(\mathcal{M})$ denotes the Hodge star operator with $\star^2= (-1)^{p(n-p)} \mathrm{Id}$ on 
$\Omega^p(\mathcal{M}_\Gamma)$. Then the Hodge Laplacian on $\Omega^p(\mathcal{M})$  is defined by
$\Delta_p=(d+d^*)^2$ .

For a flat vector bundle $E$ over $\mathcal{M}$, the above operators are extended as follows.  Let $U$ be an open subset in $\mathcal{M}$ where  $\wedge^{p-1}T^*\mathcal{M}$ and $E$ are trivial over $U$. Let $e_1,\cdots,e_r$ be a basis of flat sections of $E|_U$ where
$r$ is the rank of $E$.
Then any $\phi\in \Omega^{p-1}(U,E)$ can be written as 
\begin{equation}
\phi=\sum_{i=1}^r \phi_i\otimes e_i
\end{equation}
where $\phi_i\in \Omega^{p-1}(U)$. Now $d: \, \Omega^{p-1}(U,E) \to \Omega^p(U,E)$ is defined by
\begin{equation}\label{e:d-def-ext}
d (\phi) =\sum_{i=1}^r d\phi_i\otimes e_i.
\end{equation}
Note that this is well-defined since the flat vector bundle $E$ has a constant transition map. The operator in \eqref{e:adjoint-d} can be extended
to an operator $d^{*,\flat}:\, \Omega^p(\mathcal{M}, E) \to  \Omega^{p-1}(\mathcal{M},E) $ by
\begin{equation}\label{e:d-star-def-ext}
 d^{*,\flat}= (-1)^{np +n+1} (\star\otimes \mathrm{Id}_E)\, d \, (\star\otimes \mathrm{Id}_E) .
\end{equation} 
Here note that $(d^{*,\flat})^2=0$. 
Now an extension of the Hodge Laplacian $\Delta_p$ on  $\Omega^p(\mathcal{M},E)$ is defined by
\begin{equation}\label{e:def-lap-flat}
\Delta_p^{\flat}= (d+d^{*,\flat})^2
\end{equation} 
where $d$ and $d^{*,\flat}$ are defined in \eqref{e:d-def-ext} and \eqref{e:d-star-def-ext}.

Assuming a Hermitian metric $\langle \cdot, \cdot \rangle_E$ on $E$, we define the usual formal adjoint operator $d^*:\, \Omega^p(\mathcal{M}, E) \to  \Omega^{p-1}(\mathcal{M},E)$ extending the operator in \eqref{e:adjoint-d} by
\begin{equation}\label{e:adjoint-mu}
 d^{*}= (-1)^{np +n+1} (\star\otimes \mathrm{Id}_E)\mu^{-1} \, d \, \mu (\star\otimes \mathrm{Id}_E) .
\end{equation} 
Here $\mu: E\to E^*$ is the map defined by
\begin{equation}
\langle u,v \rangle_E = (u, \mu(v))
\end{equation}
where $(\cdot,\cdot)$ is the dual pairing. We refer  to the section 2 of \cite{MM} and the section 8 of \cite{CM} for more details of this construction. 
Now the usual
Hodge Laplacian on $\Omega^p(\mathcal{M}, E)$ is defined by
\begin{equation}
\Delta_p=(d+d^*)^2
\end{equation}
where $d$ and $d^{*}$ are defined in \eqref{e:d-def-ext} and \eqref{e:adjoint-mu}.

By the definition in \eqref{e:adjoint-mu},  $\Delta_p=\Delta_p^\flat$ when $E$ is unitarily flat, and 
for a non-unitary flat vector bundle $E$, the difference $d^*- d^{*,\flat}$ is a zero order operator. Hence, in general
$\Delta_p-\Delta_p^{\flat}$ is a first order differential operator on $\Omega^p(\mathcal{M}, E)$. For a Hermitian metric on $E$,
we can consider a $L^2$-completion of $\Omega^p(\mathcal{M}, E)$, which is denoted by $L^2(\Omega^p(\mathcal{M},E))$.
\begin{proposition}\label{p:spect-Delta}
The spectrum of non-selfadjoint operator $\Delta_p^\flat$ on $L^2(\Omega^p(\mathcal{M},E))$ is discrete and consists of generalized eigenvalues of finite multiplicities, which are contained in the set
$B_r \cup \lambda_\epsilon$
for some $r>0$ and $\epsilon>0$
where $B_r=\{z\in \mathbb{C}\, | \, |z|<r\}$ and $\Lambda_\epsilon=\{ re^{i\theta} \in\mathbb{C} \, | \, |\theta| \leq \epsilon\}$.
\end{proposition}

\begin{proof}
This follows from Theorem 8.4 and Theorem 9.3 of \cite{Sh}.
\end{proof}

 The above general construction applies to the case of the non-unitary vector bundle $E_\rho$ over a hyperbolic 3-manifold $\mathcal{M}_\Gamma$. 
In particular, we  have the operator $\Delta_p^{\flat}$ acting on $\Omega^p(\mathcal{M}_\Gamma,E_\rho)$.  The vector bundle $\wedge^p\, T^*\mathcal{M}_\Gamma$ can be given as a locally homogeneous vector bundle $E_\tau$ for the representation  $\tau= \wedge^p \mathrm{Ad}_{\frak{p}^*}$
of $K$. Hence,
\begin{equation}\label{e:ident-p-form-tilde}
\Omega^p(\tilde{\mathcal{M}_\Gamma}) \cong \left( C^\infty(G)\otimes \wedge^p \frak{p}^*  \right)^K
\end{equation}
where  $k\in K$ acts by $R(k)\otimes \wedge^p \mathrm{Ad}_{\frak{p}^*}(k)$. 
Let us denote by 
$\tilde{\Delta}_p^\flat$ the lifting of $\Delta_p^\flat$ on the universal covering space $\tilde{\mathcal{M}_\Gamma} \cong G/K$. 
With respect to \eqref{e:ident-p-form-tilde},  by Kuga's Lemma,
\begin{equation}\label{e:hodge-connection-tilde}
\tilde{\Delta}_{p}^\flat = - R(\Omega)\otimes\mathrm{Id}_{V_\rho}.
\end{equation}

\subsection{Heat kernel}
We denote by $L$ an elliptic second order differential operator acting on $C^\infty(\mathcal{M}_\Gamma, E_\tau\otimes E_\rho)$. We assume that
the spectrum  $\sigma(L)$ is discrete and consists of generalized eigenvalues of finite multiplicities and that 
$\sigma(L)$ lies in the set $B_{r} \cup \lambda_\epsilon \subset \mathbb{C}$ for some $r>0$ and $\epsilon>0$ as in Proposition 
\ref{p:spect-Delta}. 

For simplicity, first we assume that
$L$ has no zero generalized eigenvalue, that is, $0\notin \sigma(L)$. Under this condition, there exists an Agmon angle for $L$ and we can define
$L^{1/2}$ following the section 10 of \cite{Sh}. One can prove that the spectrum of ${L}^{1/2}$ lies in the subset with conditions  $\mathbb{C}$ with $\mathrm{Re}(\lambda)>\delta$ and $|\mathrm{Im}(\lambda)| < a $ for some $\delta>0$ and $a>0$. For an even test function
$\varphi(\lambda)=e^{-t\lambda^2}$  for $\lambda\in\mathbb{C}$ and $t>0$, we define $\varphi(L^{1/2})$  by 
\begin{equation}\label{e:def-phi-L12}
\varphi({L}^{1/2}):= \frac{i}{2\pi} \int_\Gamma \varphi(\lambda) \, \left({L}^{1/2} -\lambda\mathrm{Id}\, \right) ^{-1}\, d\lambda.
\end{equation}
Here $\Gamma$ is a counterclockwise oriented smooth curve given by $\Gamma_1\cup\Gamma_2\cup \Gamma_3$ where
\begin{equation*}
\Gamma_1=\{ \lambda\in \mathbb{C}\, | \,  \mathrm{Im}(\lambda)=a, \infty> \mathrm{Re} (\lambda) \geq \delta_1\}, \qquad 
\Gamma_3=\{ \lambda\in \mathbb{C}\, | \,  \mathrm{Im}(\lambda)=-a, \delta_1\leq  \mathrm{Re} (\lambda) <\infty \},
\end{equation*}
for some $\delta>\delta_1>0$ and $\Gamma_2\subset \{\lambda\in\mathbb{C}\, | \, \delta_1\leq  \mathrm{Re}(\lambda)< \delta \}$ is a simple curve connecting the finite boundary points of $\Gamma_1$ and  $\Gamma_3$.

The following construction also works for any even $\varphi$ in the space of Paley-Wiener functions, so we keep to denote our test function by $\varphi$ rather than the specific 
$\varphi(\lambda)=e^{-t\lambda^2}$.

For $f\in C^\infty(\mathcal{M}_\Gamma, E_\tau\otimes E_\rho)$, we can express $\varphi(L^{1/2})f$ 
in terms of the solution of the following wave equation
\begin{equation}\label{e:wave-eq}
\left(\frac{\partial^2 }{\partial t^2} + L\right) u =0, \qquad u(0,x)=f(x), \quad u_t(0,x)=0.
\end{equation}
By the construction in the sections IV-1 and IV-2 of \cite{Tay}, there exists a unique solution $u(t,f) \in C^\infty(\mathbb{R}\times\mathcal{M}, E_\tau\otimes E_\rho)$ of \eqref{e:wave-eq}. 
Now, by Proposition 3.2 of \cite{Mu11},
\begin{equation}\label{e:var-exp-wave}
\varphi(L^{1/2})f = \frac{1}{\sqrt{2\pi}} \int_\mathbb{R} \hat{\varphi}(t) u(t,f) \, dt 
\end{equation}
for $f\in C^\infty(\mathcal{M}_\Gamma, E_\tau\otimes E_\rho)$.

The above construction can be generalized when $L$ has a zero generalized eigenvalue. For this, following the section 2 of \cite{Mu11}, we put
\begin{equation}
\hat{L}=L(\mathrm{Id}-\Pi_0) \oplus \Pi_0
\end{equation}
where $\Pi_0$ denotes the orthogonal projection onto the generalized eigenspace $V_0$ such that there exists an integer $N_0$ with $L^{N_0} V_0=0$.
Since $\Pi_0$ is a smoothing operator, $\hat{L}$ is a pseudo-differential operator with the same symbol as $L$.  Moreover,
$\sigma(\hat{L})$ also lies in the same set $B_{r} \cup \lambda_\epsilon \subset \mathbb{C}$ and $0\notin \sigma(\hat{L})$, and $\hat{L}$ has an Agmon angle. Hence, we can define $\hat{L}^{1/2}$  as in the section 10 of \cite{Sh}. Let $V_1$ be the complementary subspace of $V_0$ which is invariant under $L$. We can also repeat the construction given in
\eqref{e:def-phi-L12} for $\hat{L}^{1/2}_1:=\hat{L}^{1/2}|_{V_1}$
 to define $\varphi(\hat{L}_1^{1/2})$.
To deal with the remaining part of $L$, putting $N:=L\Pi_0$ we define
\begin{equation}
\varphi(N^{1/2}):=\frac{1}{\sqrt{2\pi}} \int_{\mathbb{R}} \hat{\varphi}(t) \, U(t,N)\, dt
\end{equation}
where $U(t,N):= \sum_{k=0}^d \frac{(-1)^k t^{2k}}{(2k)!} N^{k}$ with $d=\mathrm{dim}(V_0)$. Then, combining these constructions, we define
\begin{equation}
\varphi(L^{1/2}):= \varphi(\hat{L}^{1/2}_1)(\mathrm{Id}-\Pi_0) +\varphi(N^{1/2})\Pi_0.
\end{equation}
By Proposition 3.2 of \cite{Mu11}, the expression as in \eqref{e:var-exp-wave} holds even when $0\in \sigma(L)$. 

\begin{proposition}\label{p:wave-kernel}
For $L=\Delta_p^\flat$ and  $\varphi(\lambda)=e^{-t\lambda^2}$, the heat
operator $e^{-tL}:=\varphi(L^{1/2})$ is of trace class operator with the smooth kernel
\begin{equation}\label{e:K-kernel}
K_\varphi(\Gamma g_1K, \Gamma g_2K) =\sum_{\gamma\in \Gamma} H_\varphi(g_1^{-1}\gamma g_2)\otimes \rho(\gamma)
\end{equation}
where $\Gamma g_1 K, \Gamma g_2K\in \mathcal{M}_\Gamma\cong \Gamma\backslash G/K$ and $H_\varphi: G\to \mathrm{End}(V_\tau)$ is a $C^\infty$-function which satisfies 
$H_\varphi(k_1g k_2) =\tau(k_1)\circ H_\varphi(g)\circ \tau(k_2)$ for $k_1,k_2\in K$.
\end{proposition}

\begin{proof}
Basically we follow the construction given in the sections 3 and 4 of \cite{Mu11} with some modification since $\hat{\varphi}$ does not have a compact support for $\varphi(\lambda)=e^{-t\lambda^2}$. 

First, one can show that $\varphi(L^{1/2})$ is of trace class with a smooth kernel $K_\varphi$ by a standard argument as in Lemma 2.4 of \cite{Mu11}.
To derive the expression in \eqref{e:K-kernel}, we consider the liftings of $u(t,x,f)$ and $f$ satisfying \eqref{e:wave-eq} to $\tilde{\mathcal{M}}_\Gamma$, which we
denote by $\tilde{u}(t,\tilde{x}, f)$ and $\tilde{f}$ respectively. Then, for the operator
$\tilde{L}^\flat=\tilde{\Delta}_p^\flat$ over $\tilde{\mathcal{M}}_\Gamma$,  the lifted  solution $\tilde{u}(t,f)$ satisfies
\begin{equation}\label{e:wave-eq-lif}
\left(\frac{\partial^2}{\partial t^2} +\tilde{ L}^\flat\right) \tilde{u}(t,f) =0, \qquad \tilde{u}(0,f)=\tilde{f}, \quad \tilde{u}_t(0,f)=0.
\end{equation}
By the energy estimate given in the chapter 2 of \cite{Tay}, the solutions of $(\frac{\partial^2 }{\partial t^2} +\tilde{ L}^\flat){u}=0 $ have finite propagation speed. 
This implies that for every $\psi\in C^\infty(\tilde{\mathcal{M}}_\Gamma, \tilde{E}_\tau\otimes \tilde{E}_\rho)$, the wave equation
\begin{equation}\label{e:wave-eq-lif2}
\left(\frac{\partial^2 }{\partial t^2} +\tilde{ L}^\flat\right) {u}(t,\psi) =0, \qquad {u}(0,\psi)=\psi, \quad {u}_t(0,\psi)=0
\end{equation}
has a unique solution. Hence,
\begin{equation}\label{e:unique-sol}
\tilde{u}(t,f)= u(t,\tilde{f}).
\end{equation}

Since $\tilde{\mathcal{M}}_\Gamma\cong G/K$ with $G=\mathrm{SL}(2,\mathbb{C})$ and $K=\mathrm{SU}(2)$, we apply some harmonic analysis for $(G,K)$
to obtain more explicit expression of \eqref{e:unique-sol}.
First, recall that
$[\pi|_K:\tau] \leq 1$ for any $\tau\in\hat{K}$, $\pi\in \hat{G}$. For $\pi\in\hat{G}(\tau) =\{ \pi\in \hat{G}\, :\, [\pi|_K: \tau] =1 \}$,
the $\tau$-isotypical subspace $\mathcal{H}_\pi(\tau)$ of $\tau$ in $\mathcal{H}_\pi$ can be identified with $V_\tau$. 
Define $\tau$-spherical function $\Phi^\pi_\tau$ on $G$ by
\begin{equation}\label{e:e:def-tau-sph}
\Phi^\pi_\tau(g)=P_\tau \pi(g) P_\tau 
\end{equation}
for $g\in G$ where $P_\tau$ denotes the orthogonal projection of $\mathcal{H}_\pi$ onto
the $\tau$-isotypical subspace $\mathcal{H}_\pi(\tau)$.
Moreover, we have the following identification
\begin{equation}\label{e:hom-space}
C^\infty(G/K, \tilde{E}_\tau) \cong C^\infty(G;\tau)
\end{equation}
where
\begin{equation}\label{e:equiv-space}
C^\infty(G; \tau)=\left\{ f : C^\infty(G, V_\tau) \, | \, f(gk)=\tau(k^{-1})f(g), g\in G, k\in K\right\}.
\end{equation} 
Then, for $f^\pi_{\tau,v}:= \Phi^\pi_\tau(g^{-1})(v) \in C^\infty(G;\tau)$ with $v\in V_\tau$,  we define
\begin{equation}\label{e:L-action}
\tilde{L}\, f^\pi_{\tau,v} = -\pi(\Omega)  f^\pi_{\tau,v}.
\end{equation}
 Hence, the unique solution $u(t,x, f^\pi_{\tau,v})$ with $ u(0)= f^\pi_{\tau,v}$ of the corresponding equation with $\tilde{L}$ to \eqref{e:wave-eq-lif2}
is given by
\begin{equation}\label{e:wave-sol}
u(t,x,f^\pi_{\tau,v}) =\cos\left( t\sqrt{ -\pi(\Omega) }\right)\, f^\pi_{\tau,v}.
\end{equation}
This immediately implies
\begin{equation}\label{e:wave-action-pi}
\begin{split}
&\varphi\left(\tilde{L}^{1/2}\right) f^\pi_{\tau,v}= \varphi\left( \sqrt{ -\pi(\Omega) } \right) f^\pi_{\tau,v}\\
=&\frac{1}{\sqrt{2\pi}} \int_{\mathbb{R}} \hat{\varphi}(t)\, \cos\left(t\sqrt{ -\pi(\Omega) } \right) f^\pi_{\tau,v}  \, dt
= \frac{1}{\sqrt{2\pi}} \int_{\mathbb{R}} \hat{\varphi}(t)\, u(t,x,f^\pi_{\tau,v})  \, dt.
\end{split}
\end{equation}
By the same way as the proposition 2.2 of \cite{P09}, one can show that $\varphi(\tilde{L}^{1/2})$ has a smooth kernel
$\tilde{K}_\varphi\in C^\infty(\tilde{\mathcal{M}}_\Gamma\times \tilde{\mathcal{M}}_\Gamma, \mathrm{Hom}(\tilde{E}_\tau, \tilde{E}_\tau))$
such that for $\psi\in C^\infty(\tilde{\mathcal{M}}_\Gamma, \tilde{E}_\tau)$,
\begin{equation}
\frac{1}{\sqrt{2\pi}} \int_\mathbb{R} \hat{\varphi}(t) u(t, \tilde{x},\psi)\, dt = \int_{\tilde{\mathcal{M}}_\Gamma} \tilde{K}_\varphi(\tilde{x}, \tilde{y} ) \psi(\tilde{y})\, d\tilde{y}.
\end{equation}
Since $\varphi(\tilde{L}^{1/2})$ is a $G$-invariant integral operator, its kernel $\tilde{K}_\varphi$ satisfies
\begin{equation}\label{e:kernel-inv}
\tilde{K}_\varphi(g\tilde{x}, g\tilde{y}) =\tilde{K}_\varphi(\tilde{x},\tilde{y}), \qquad \text{for} \ \ g\in G.
\end{equation}
With respect to \eqref{e:hom-space}, the kernel $\tilde{K}_\varphi$ can be identified with a $C^\infty$-function
$H_\varphi:G\to \mathrm{End}(V_\tau)$,
which satisfies 
\begin{equation}\label{e:h-equiv}
H_\varphi(k_1g k_2) =\tau(k_1)\circ H_\varphi(g)\circ \tau(k_2) \qquad \text{for} \  \ k_1,k_2\in K .
\end{equation}
Then $\varphi(\tilde{L}^{1/2})$ acts by convolution
\begin{equation}
\left(\varphi(\tilde{L}^{1/2}) f \right)(g_1)= \int_G H_\varphi(g_1^{-1} g_2) f(g_2)\, dg_2.
\end{equation}
Recalling the relations of $\Delta_p^\flat$ with $\tilde{L}^\flat=\tilde{\Delta}^\flat_p$ and  with $\tilde{L}$, 
one can see that the kernel  $K_\varphi$ of $\varphi\left( L^{1/2}\right)$  has the form given in \eqref{e:K-kernel}.
This completes the proof.
\end{proof}

By Proposition \ref{p:wave-kernel}, applying the Lidskii's theorem
(see Theorem 8.4 of \cite{GK}), and the definition of an integral kernel, we have

\begin{corollary}
For $L=\Delta_p^\flat$ and  $\varphi(\lambda)=e^{-t\lambda^2}$, the following equalities hold
\begin{equation}\label{e:spectral-side}
\begin{split}
\sum_{\lambda\in \mathrm{spec}(L)} m(\lambda) e^{-t\lambda}=&\ \mathrm{Tr}\,\left( e^{-tL}\right) \\ 
=& \int_{\mathcal{M}_\Gamma} \mathrm{tr}\, K_{\varphi}(m,m)\, dm = \sum_{\gamma\in\Gamma} \mathrm{tr}\, \rho(\gamma) \int_{\Gamma\backslash G} \, h_\varphi(g^{-1}\gamma g)\, d{\dot{g}}
\end{split}
\end{equation}
where $m(\lambda)$ denotes the multiplicity of the generalized eigenvalue $\lambda$ of $L$  and $h_\varphi=\mathrm{tr}H_\varphi$. 
\end{corollary}

\subsection{Selberg trace formula}

\begin{proposition}
The following equalities hold
\begin{equation}\label{e:heat-trace0}
\begin{split}
\mathrm{Tr}(e^{-t(\Delta^\flat_0-1)})=\, & {\mathrm{dim}(V_\rho)\,\mathrm{Vol}(\mathcal{M}_\Gamma)} \frac{1}{4\pi^2} \int_{\mathbb{R}} \, \lambda^2\,
e^{-t\lambda^2}\, d\lambda\\
&+\sum_{[\gamma]\neq [e]} \frac{\ell_{\gamma_0}\,  \mathrm{tr}\,\rho(\gamma)}{ D(\gamma)} 
\frac{1}{\sqrt{4\pi t}}  e^{- \ell_\gamma^2/4t},
\end{split}
\end{equation}
\begin{equation}\label{e:heat-trace}
\begin{split}
\mathrm{Tr}(e^{-t\Delta^\flat_1}) - \mathrm{Tr}(e^{-t\Delta^\flat_0}) 
=\, &  {\mathrm{dim}(V_\rho)\,\mathrm{Vol}(\mathcal{M}_\Gamma)} \frac{1}{2\pi^2} \int_{\mathbb{R}} \, \left(\lambda^2+ 1\right)\, e^{-t\lambda^2} 
  \, d\lambda\\
&+ \sum_{[\gamma]\neq [e]} \frac{\ell_{\gamma_0}\,  \mathrm{tr}\,\rho(\gamma)}{ D(\gamma)} 
\left(e^{i\theta_\gamma}+ e^{-i\theta_\gamma} \right)\frac{1}{\sqrt{4\pi t}}  e^{- \ell_\gamma^2/4t} 
\end{split}
\end{equation}
where  $D(\gamma)=  e^{\ell_\gamma} |(1-e^{-(\ell_\gamma+i\theta_\gamma)})|^2$ defined in \eqref{e:def-D-gamma} and the sums on the right hand sides run over the set of conjugacy classes in $\Gamma$ of hyperbolic elements in $\Gamma$.
\end{proposition}

\begin{remark} The same formulae as \eqref{e:heat-trace0} and \eqref{e:heat-trace} are well known for the selfadjoint operators $\Delta_p$,  which can be derived from Theorem 6.7 of \cite{Wall} easily.
Note that we used the following equality
for the first terms on the right hand sides of \eqref{e:heat-trace0} and \eqref{e:heat-trace},
\begin{equation}\label{e:comp-volume}
\mathrm{Vol}(\mathcal{M}_\Gamma)= \pi \mathrm{Vol}(\Gamma\backslash G).
\end{equation}
This can be derived by the same way as the equation (4.31) of \cite{GP}. 
\end{remark}

\begin{proof}
Since the proof is essentially the same, we prove only the equality \eqref{e:heat-trace}.
From the right hand side of \eqref{e:spectral-side} with $\varphi(\lambda)=e^{-t\lambda^2}$ and using \eqref{e:ind-rest-map}, we can proceed as in the original way of Selberg to obtain the following equality in our context.
For more details about this we refer to the section 4 of \cite{P09}.
\begin{equation}\label{e:Selberg-trace-1}
\begin{split}
&\sum_{\lambda\in \mathrm{spec}(\Delta^\flat_1)} m(\lambda) e^{-t\lambda} - \sum_{\lambda\in \mathrm{spec}(\Delta^\flat_0)} m(\lambda) e^{-t\lambda} \\
=&\  {\mathrm{dim}(V_\rho)\,\mathrm{Vol}(\mathcal{M}_\Gamma)} \frac{1}{2\pi^2} \int_{\mathbb{R}} \, \Theta_{2,\lambda}(h_\varphi)
  \left(\lambda^2+ 1\right)\, d\lambda\\
&+\ \sum_{[\gamma]\neq [e]} \frac{\ell_{\gamma_0}\,  \mathrm{tr}\,\rho(\gamma)}{ D(\gamma)} 
\left( e^{i\theta_\gamma}+ e^{-i\theta_\gamma}\right)
\frac{1}{2\pi} \int_{\mathbb{R}}\, \Theta_{2,\lambda}(h_\varphi) \, e^{-i \ell_\gamma\lambda}\, d\lambda.
\end{split}
\end{equation}
Here we denote by $\Theta_{k,\lambda}$ the character of the induced representation $(\pi_{k,\lambda}, \mathcal{H}_{k,\lambda})$ where
\begin{equation}\label{e:indcued-rep}
\pi_{k,\lambda}=\mathrm{Ind}^G_{MAN}(\sigma_k\otimes e^{i\lambda}\otimes 1).
\end{equation}
By \eqref{e:wave-action-pi} and the equality $\pi_{2,\lambda}(\Omega)=-\lambda^2$, 
\begin{equation}\label{e:theta-varphi}
\Theta_{2,\lambda}(h_\varphi) = \varphi \left( \sqrt{ -\pi_{2,\lambda}(\Omega) } \right) = \varphi\left( \lambda \right) = e^{-t\lambda^2}.
\end{equation}
Here we used that $\varphi$ is even for th second equality.
Hence, the first term of the right hand side of \eqref{e:Selberg-trace-1} has the expression as in \eqref{e:heat-trace}.
For the second term of the right hand side of \eqref{e:Selberg-trace-1}, we use the following equality
\begin{equation}
\frac{1}{2\pi}\int^\infty_{-\infty} e^{-t\lambda^2} e^{-i \ell \lambda}\, d\lambda=\frac{1}{\sqrt{4\pi t}} e^{-\ell^2/4t}
\end{equation}
to obtain the expression as in \eqref{e:heat-trace}. This completes the proof of \eqref{e:heat-trace}.

\end{proof}


\section{Zeta functions} 

For a discrete subgroup $\Gamma$, the \emph{critical exponent} $\delta(\Gamma)$ is defined by
\begin{equation}
\delta(\Gamma)= \mathrm{inf} \left\{ s\ \big{\vert} \ \sum_{\gamma\in\Gamma} e^{-s \ell_\gamma} < \infty \right\}.
\end{equation}
It is known that $\delta(\Gamma)<2$ for a cocompact torsion free discrete subgroup $\Gamma$ (see Theorem 3.14.1 of \cite{Mar}). 
Hence, we have 
\begin{equation}\label{e:inf-prod}
 \prod_{[\gamma]_{\mathrm{p}}} \prod_{k=1}^\infty(1-e^{-sk\ell_\gamma}) < \infty \qquad \text{ for\ \  $\mathrm{Re}(s)>2$}.
\end{equation}
Here and from now on,  the product notation with $[\gamma]_{\mathrm{p}}$ always means that the product runs over  the set of conjugacy classes in $\Gamma$ of the \emph{primitive} hyperbolic elements in $\Gamma$. First we put
\begin{equation}
R(\sigma_k,s)= \prod_{{[\gamma]_{\mathrm{p}} }} \big(1 - e^{\frac{k}{2}(i\theta_\gamma)}\ e^{-s\, l_{\gamma}} \big)  \qquad \text{for\ \  $\mathrm{Re}(s) >2$}.
\end{equation}
It is well known that $R(\sigma_k,s)$ has a meromorphic extension over $\mathbb{C}$ (see \cite{BO}).   

For a representation $(\chi, V_\chi)$ of  $\Gamma$, 
the \emph{Ruelle zeta function} $R_\chi(s)$ attached to $\chi$ is defined by
\begin{equation}\label{e:def-ruelle}
R_\chi(s) = \prod_{{[\gamma]_{\mathrm{p}}}} \mathrm{det}\big(\mathrm{Id} - \chi(\gamma)\ e^{-s\, l_{\gamma}} \big) \qquad \ \text{for \ \ $\mathrm{Re} (s) \gg 0 $}.
\end{equation}
We are interested in the case $\chi=\rho|_\Gamma$ where $\rho=\rho_m$ is a representation of $G$. 
Now we have

\begin{proposition}\label{p:ruelle-dec-ruelle} For the restriction to $\Gamma$ of the representation $\rho_m$ of $G$, 
the Ruelle zeta function attached to $\rho_m$ has the following expression
\begin{equation}\label{e:ruelle-dec-ruelle}
R_{\rho_m}(s)= \prod_{l=0}^m R\left(\sigma_{{m}-2l}, s-\small{\frac{m}{2}}+l\right) \qquad \text{for \quad $\mathrm{Re} (s) >2+\frac{m}{2}$},
\end{equation}
and $R_{\rho_m}(s)$ has a meromorphic extension over $\mathbb{C}$.
\end{proposition}
\begin{proof} The equality \eqref{e:ruelle-dec-ruelle} follows from \eqref{e:gamma-conj} and \eqref{e:decomp-rho} easily. Then, the meromorphic extension of $R_{\rho_m}(s)$
over $\mathbb{C}$
follows from the one of $R(\sigma_k,s)$.
\end{proof}

For a representation $(\chi, V_\chi)$ of  $\Gamma$,  the \emph{Selberg zeta function}
$Z_\chi(\sigma_k,s)$ attached to $\chi$ is defined by
\begin{equation*}
Z_\chi(\sigma_k, s) = \prod_{[\gamma]_{\mathrm{p}}} \prod_{p \geq 0,q\geq 0} \mathrm{det}\big(\mathrm{Id}- \chi(\gamma) e^{\frac{k}2i\theta_\gamma} e^{-p(l_\gamma+i\theta_\gamma)} e^{-q(l_\gamma-i\theta_\gamma)} e^{-s\, l_\gamma}\big) \qquad \ \text{for \ \ $\mathrm{Re} (s) \gg 0$}.
\end{equation*}
For the trivial $\chi$, we denote it by $Z(\sigma_k,s)$, and  the convergence 
of  $Z(\sigma_k,s)$ for $\mathrm{Re}(s)>2$ follows from  \eqref{e:inf-prod}.
It is also well known that $Z(\sigma_k,s)$ has a meromorphic extension over $\mathbb{C}$ (see \cite{BO, GP}). For $\chi=\rho_m$,
the Selberg zeta function $Z_{\rho_m}(\sigma_k,s)$ has the following expression by \eqref{e:gamma-conj} and \eqref{e:decomp-rho}
\begin{equation}\label{e:relation-Selberg-zeta}
Z_{\rho_m}(\sigma_k,s)= \prod^{m}_{l=0} Z\left(\sigma_{m-2l+k}, s-\frac{m}{2}+l\right) \qquad \text{for} \ \ \mathrm{Re}(s)> 2+\frac{m}{2},
\end{equation}
which has the meromorphic extension over $\mathbb{C}$.

In the following two propositions, we will derive two expressions for $R_{\rho_m}(s)$ in terms of the Selberg zeta functions. These two formulae and their relationship are the starting point for the proofs of main results of this paper. 

\begin{proposition}\label{p:ruelle-decomp1} The following equality holds over $\mathbb{C}$,
\begin{equation}\label{e:ruelle-decomp1}
R_{\rho_m}(s)=\frac{ Z(\sigma_m, s-\frac{m}{2}) \, Z(\sigma_{-m}, s+\frac{m}{2}+2)}
{Z(\sigma_{m+2}, s-\frac{m}{2}+1) \, Z(\sigma_{-(m+2)}, s+\frac{m}{2}+1)}.
\end{equation}
\end{proposition}

\begin{proof}
Since the Selberg zeta function $Z(\sigma_k,s)$ attached to $\sigma_k$ has a meromorphic  extension over $\mathbb{C}$, it is sufficient to show the equality \eqref{e:ruelle-decomp1} over a domain where both sides converge absolutely. Over such a domain, we have
\begin{equation}\label{e:ruelle-dec-proof1}
\begin{split}
Z(\sigma_m, s-\frac{m}{2}) =& \prod_{[\gamma]_{\mathrm{p}}} \prod_{p \geq 0,q\geq 0} \big(1- e^{\frac{m}2i\theta_\gamma} e^{-p(l_\gamma+i\theta_\gamma)} e^{-q(l_\gamma-i\theta_\gamma)} e^{-(s-\frac{m}{2})\, l_\gamma}\big)\\
=&  \prod_{[\gamma]_{\mathrm{p}}} \prod_{p \geq 0,q\geq 0} \big(1- e^{\frac{m}2(\ell_\gamma+i\theta_\gamma)} e^{-p(l_\gamma+i\theta_\gamma)} e^{-q(l_\gamma-i\theta_\gamma)} e^{-s\, l_\gamma}\big).
\end{split}
\end{equation}
Similarly we have
\begin{equation}\label{e:ruelle-dec-proof2}
\begin{split}
Z(\sigma_{m+2}, s-\frac{m}{2}+1) =& \prod_{[\gamma]_{\mathrm{p}}} \prod_{p \geq 0,q\geq 0} \big(1- e^{(\frac{m}2+1)i\theta_\gamma} e^{-p(l_\gamma+i\theta_\gamma)} e^{-q(l_\gamma-i\theta_\gamma)} e^{-(s-\frac{m}{2}+1)\, l_\gamma}\big)\\
=&  \prod_{[\gamma]_{\mathrm{p}}} \prod_{p \geq 0,q\geq 1} \big(1- e^{\frac{m}2(\ell_\gamma+i\theta_\gamma)} e^{-p(l_\gamma+i\theta_\gamma)} e^{-q(l_\gamma-i\theta_\gamma)} e^{-s\, l_\gamma}\big).
\end{split}
\end{equation}
By \eqref{e:ruelle-dec-proof1} and \eqref{e:ruelle-dec-proof2},
\begin{equation}\label{e:ruelle-dec-proof3}
\frac{Z(\sigma_m, s-\frac{m}{2})}{Z(\sigma_{m+2}, s-\frac{m}{2}+1)} =\prod_{[\gamma]_{\mathrm{p}}} \prod_{p \geq 0} \big(1- e^{\frac{m}2(\ell_\gamma+i\theta_\gamma)} e^{-p(l_\gamma+i\theta_\gamma)}  e^{-s\, l_\gamma}\big).
\end{equation}
By the same way,
\begin{equation}\label{e:ruelle-dec-proof4}
\frac{Z(\sigma_{-m}, s+\frac{m}{2}+2)}{Z(\sigma_{-(m+2)}, s+\frac{m}{2}+1)} =\prod_{[\gamma]_{\mathrm{p}}} \prod_{p \geq 1} \big(1- e^{-\frac{m}2(\ell_\gamma+i\theta_\gamma)} e^{-p(l_\gamma+i\theta_\gamma)}  e^{-s\, l_\gamma}\big)^{-1}.
\end{equation}
By \eqref{e:ruelle-dec-proof3} and \eqref{e:ruelle-dec-proof4}, 
\begin{equation}
\begin{split}
\frac{ Z(\sigma_m, s-\frac{m}{2}) \, Z(\sigma_{-m}, s+\frac{m}{2}+2)}
{Z(\sigma_{m+2}, s-\frac{m}{2}+1) \, Z(\sigma_{-(m+2)}, s+\frac{m}{2}+1)}
=\ &\prod_{[\gamma]_{\mathrm{p}}} \prod_{k= 0}^m \big(1- e^{(\frac{m}2-k)(\ell_\gamma+i\theta_\gamma)}   e^{-s\, l_\gamma}\big)\\
=\ & \prod_{k=0}^m R\left(\sigma_{{m}-2k}, s-\small{\frac{m}{2}}+k\right).
\end{split}
\end{equation} 
Combining this and Proposition \ref{p:ruelle-dec-ruelle} completes the proof.
\end{proof}

\begin{proposition}\label{p:ruelle-decomp2}
The following equality holds over $\mathbb{C}$,
\begin{equation}\label{e:ruelle-decomp2}
R_{\rho_m}(s)= \frac{ Z_{\rho_m}(\sigma_0, s) \, Z_{\rho_m}(\sigma_0,s+2)}{Z_{\rho_m}(\sigma_2,s+1)\, Z_{\rho_m}(\sigma_{-2},s+1)}.
\end{equation}
\end{proposition}
\begin{proof}
By \eqref{e:relation-Selberg-zeta}, we have
\begin{equation}\label{e:rat-Sel-2}
\frac{Z_{\rho_m}(\sigma_0,s)}{Z_{\rho_m}(\sigma_{-2}, s+1)} = \frac{Z(\sigma_m, s-\frac{m}{2})}{Z(\sigma_{-(m+2)}, s+\frac{m}{2}+1)},
\end{equation}
and
\begin{equation}\label{e:rat-Sel-1}
\frac{Z_{\rho_m}(\sigma_0, s+2)}{Z_{\rho_m}(\sigma_2, s+1)} = \frac{Z(\sigma_{-m}, s+\frac{m}{2}+2)}{Z(\sigma_{m+2}, s-\frac{m}{2} +1)}.
\end{equation}
Combining \eqref{e:rat-Sel-2}, \eqref{e:rat-Sel-1} and \eqref{e:ruelle-decomp1} completes the proof.
\end{proof}

Now let us introduce one of the main ingredients of this paper. This is defined by the infinite products of the Ruelle zeta functions attached to 
$\sigma_{-m}$ for even or odd integers $m\in\mathbb{N}$:
\begin{align}
F_n(s) =& \prod_{k=n}^\infty R(\sigma_{-2k}, s+k) \qquad \qquad \qquad \, \text{for \ \ $\mathrm{Re} (s) > 2-n$, \qquad $n\in \mathbb{N}$  },\label{e:def-Fs}\\
G_n(s)=& \prod_{k=n}^\infty R(\sigma_{-(2k+1)}, s+k+\frac12) \qquad \ \ \text{for \ \   $\mathrm{Re} (s) > \frac 32-n$, \qquad $n\in\mathbb{N}\cup\{0\}$ }.\label{e:def-Gs}
\end{align}
The convergences of $F_n(s)$, $G_n(s)$ over each half planes follow from \eqref{e:inf-prod}.
Note that a spin structure is
involved in the construction of $G_n(s)$.
The relations of $F_n(s)$, $G_n(s)$ with other zeta functions are given as follows.

\begin{proposition}\label{p:F-rel} The functions $F_n(s)$, $G_n(s)$ have meromorphic extensions over $\mathbb{C}$ and they satisfy the following relations with Selberg
zeta functions:
\begin{align}
F_n(s)=& \frac{Z(\sigma_{-2n}, s+n)}{Z(\sigma_{-2(n-1)},s+n+1)}, \label{e:F-exp}\\
G_n(s)=& \frac{Z(\sigma_{-(2n+1)}, s+n+\frac12)}{Z(\sigma_{-(2n-1)}, s+n+\frac 32)}.\label{e:F-exp1}
\end{align}
\end{proposition}

\begin{proof}
The  equality of \eqref{e:F-exp} over the absolute convergence domain $\mathrm{Re}(s)>2-n$ implies the meromorphic extension of $F_n(s)$ over $\mathbb{C}$. Hence it suffices to show the equality over the convergence domain.
Over the absolute convergence domain, we have
\begin{equation}\label{e:p1}
\begin{split}
F_n(s)=& \prod_{k=n}^\infty \prod_{[\gamma]_{\mathrm{p}}} (1 - e^{-k(\ell_\gamma+i\theta_\gamma)} e^{-s\ell_\gamma} )\\
    =&   \prod_{[\gamma]_{\mathrm{p}}} \prod_{p\geq 0, q\geq0} 
\frac{ (1 - e^{-(p+n)(\ell_\gamma+i\theta_\gamma)} e^{-q(\ell_\gamma-i\theta_\gamma)} e^{-s\ell_\gamma}) }
{(1 - e^{-(p+n)(\ell_\gamma+i\theta_\gamma)} e^{-(q+1)(\ell_\gamma-i\theta_\gamma)} e^{-s\ell_\gamma})}\\
    =& \frac{Z(\sigma_{-2n}, s+n)}{Z(\sigma_{-2(n-1)},s+n+1)}.
\end{split}
\end{equation}

As above the meromorphic extension of $G_n(s)$ can be proved by the following equality over the absolute convergence domain $\mathrm{Re}(s)>\frac 32-n$,
\begin{equation*}
\begin{split}
G_n(s)= &\prod_{k=n}^\infty \prod_{[\gamma]_{\mathrm{p}}} (1 - e^{-(k+1/2)(\ell_\gamma+i\theta_\gamma)} e^{-s\ell_\gamma}  )\\
=& \prod_{[\gamma]_{\mathrm{p}}} \prod_{p\geq 0, q\geq0}
\frac{ (1 - e^{-(p+n+1/2)(\ell_\gamma+i\theta_\gamma)} e^{-q(\ell_\gamma-i\theta_\gamma)} e^{-s\ell_\gamma}) }
{(1 - e^{-(p+n+1/2)(\ell_\gamma+i\theta_\gamma)} e^{-(q+1)(\ell_\gamma-i\theta_\gamma)} e^{-s\ell_\gamma})}\\
= &\frac{Z(\sigma_{-(2n+1)}, s+n+\frac12)}{Z(\sigma_{-(2n-1)}, s+n+\frac 32)}.
\end{split}
\end{equation*}
\end{proof}


By Propositions \ref{p:ruelle-decomp1} and \ref{p:F-rel}, we have

\begin{corollary} The following equality holds for $s\in\mathbb{C}$ and $n\in\mathbb{N}$,
\begin{align}
F_n(s)^2R_{\rho_{2(n-1)}}(s)=&\frac{Z(\sigma_{2(n-1)}, s-n+1) \, Z(\sigma_{-2n},s+n)}{Z(\sigma_{-2(n-1)}, s+n+1)\, Z(\sigma_{2n}, s-n+2)}, \label{e:F-Ruelle-rel}\\
G_n(s)^2R_{\rho_{2n-1}}(s)=&\frac{Z(\sigma_{2n-1}, s-n+\frac12) \, Z(\sigma_{-(2n+1)},s+n+\frac12)}{Z(\sigma_{2n+1}, s-n+\frac32)\, Z(\sigma_{-(2n-1)}, s+n+\frac32)}.
\label{e:G-Ruelle-rel}
\end{align}
\end{corollary}

 The following theorem can be considered as a primitive form of Theorem \ref{t:main-theorem} and Theorem \ref{t:odd-case}.

\begin{theorem} \label{t:F-rel} The functions $F_n(s)^2R_{\rho_{2(n-1)}}(s)$ is regular at $s=0$ and
\begin{equation}\label{e:F-rel-prim}
\begin{split}
&F_n(s)^4 R_{\rho_{2(n-1)}}(s)^2 \Big|_{s=0} \\
=& \exp\left(\, -\frac{2}{\pi} \left(2n^2-2n+\frac13\right) \mathrm{Vol}(\mathcal{M}_\Gamma) 
-2i\pi \left(\eta(D(\sigma_{2n})) -\eta(D(\sigma_{2(n-1)}))\right) \, \right).
\end{split}
\end{equation}

The function $G_n(s)^2 R_{\rho_{2n-1}}(s)$ is regular at $s=0$ and
\begin{equation}\label{e:G-rel-prim}
\begin{split}
&G_n(s)^4 R_{\rho_{2n-1}}(s)^2 \Big|_{s=0} \\
=&  \exp\left(\, -\frac{2}{\pi} \left(2n^2-\frac16\right) \mathrm{Vol}(\mathcal{M}_\Gamma) 
-2i\pi \left(\eta(D(\sigma_{2n+1})) -\eta(D(\sigma_{2n-1}))\right) \, \right).
\end{split}
\end{equation}

\end{theorem}

\begin{remark} \label{r:reg-Ruelle-zero}
By definitions in \eqref{e:def-Fs}, \eqref{e:def-Gs}, and equalities \eqref{e:F-rel-prim}, \eqref{e:G-rel-prim}, one can see that
$R_{\rho_m}(s)$ is regular and has a finite value
at $s=0$ for $m\geq 3$.
\end{remark}

\begin{proof}
We prove only for the case with $F_n(s)$ since the proof for the other case with $G_n(s)$ is the same.
We start by rewriting \eqref{e:F-Ruelle-rel} as follows,
\begin{equation}\label{e:re-F-ruelle-rel}
\begin{split}
F_n(s)^2R_{\rho_{2(n-1)}}(s)= & 
\frac{Z(\sigma_{2(n-1)}, s-n+1)}{Z(\sigma_{-2(n-1)},-s+n+1)}\, \frac{Z(\sigma_{-2(n-1)}, -s+n+1)}{Z(\sigma_{-2(n-1)}, s+n+1)} \\
\cdot & \frac{Z(\sigma_{-2n}, -s+n)}{Z(\sigma_{2n}, s-n+2)}\, \frac{Z(\sigma_{-2n}, s+n)}{Z(\sigma_{-2n}, -s+n)}.
\end{split}
\end{equation}
By Theorem 3.18 of \cite{BO},
\begin{equation}\label{e:funct-Selberg-zeta}
Z(\sigma_k,1+s) = e^{i\pi \eta(D(\sigma_k))} \exp\left( \frac{\mathrm{Vol}(\mathcal{M}_\Gamma)}{\pi} \left(\frac{1}{3} s^3- \frac{k^2}{4}s \right) \right) 
Z(\sigma_{-k},1-s).
\end{equation}
By \eqref{e:re-F-ruelle-rel} and \eqref{e:funct-Selberg-zeta},
\begin{equation}\label{e:re-F-ruelle-rel2}
\begin{split}
&F_n(s)^2 R_{\rho_{2(n-1)}}(s)\Big|_{s=0}\\
 =&\exp\left(\, -\frac{1}{\pi} \left(2n^2-2n+\frac13\right) \mathrm{Vol}(\mathcal{M}_\Gamma) 
-i\pi \left(\eta(D(\sigma_{2n})) -\eta(D(\sigma_{2(n-1)}))\right) \, \right)\\
\cdot & \left(\,  \frac{Z(\sigma_{-2(n-1)}, -s+n+1)}{Z(\sigma_{-2(n-1)}, s+n+1)}\,  \frac{Z(\sigma_{-2n}, s+n)}{Z(\sigma_{-2n}, -s+n)}\, \right) \Big|_{s=0}.
\end{split}
\end{equation}
The last part evaluated at $s=0$ on the right hand side of \eqref{e:re-F-ruelle-rel2} need not to be  $1$ since each factor given by the Selbeg zeta
function may have a zero or a pole at $s=0$. Hence, the concerning part is $1$ or $-1$ in general. 
We remove this ambiguity by taking the square to obtain \eqref{e:F-rel-prim}.

\end{proof}

\begin{theorem} The following equalities hold
\begin{align}
F_1(s)^2 R_{\rho_0}(s) \Big |_{s=0} =&-\exp \Big(-\frac{1}{3\pi}\left(\, \mathrm{Vol}(\mathcal{M}_\Gamma)+i3\pi^2 \eta(D(\sigma_2))\, \right)\Big), \label{e:F-rel}\\
G_0(s)^4 \Big|_{s=0}=& \exp \Big(\frac{1}{3\pi}\left(\,
\mathrm{Vol}(\mathcal{M}_\Gamma)-i12\pi^2 \eta({D}({\sigma_1}))\right) \Big)\label{e:G-rel}.
\end{align}
\end{theorem}

\begin{remark}\label{r:FG}
For the equality \eqref{e:G-rel}, the Ruelle zeta function does not appear on the left hand side of it. This is because the corresponding
term is formally $R_{\rho_{-1}}(0)$, which can be understood to be $1$ putting $m=-1$ on the right hand side of \eqref{e:ruelle-decomp1}. 
\end{remark}

\begin{proof}
The equality \eqref{e:re-F-ruelle-rel2} for $n=1$ is written as
\begin{equation}\label{e:rel-F-ruelle-rel2-1}
\begin{split}
&F_1(s)^2 R_{\rho_{0}}(s)\Big|_{s=0}\\
 =&\exp\left(\, -\frac{1}{3\pi}  \mathrm{Vol}(\mathcal{M}_\Gamma) 
-i\pi \eta(D(\sigma_{2}))  \, \right)
\cdot  \left(\,  \frac{Z(\sigma_0, -s+2)}{Z(\sigma_0, s+2)}\,  \frac{Z(\sigma_{-2}, s+1)}{Z(\sigma_{-2}, -s+1)}\, \right) \Big|_{s=0}
\end{split}
\end{equation}
noting $\eta(D(\sigma_0))=0$. By \eqref{e:funct-Selberg-zeta},
\begin{equation}
Z(\sigma_0, \pm s) = \exp\left( \frac{\mathrm{Vol}(\mathcal{M}_\Gamma)}{\pi} \frac{1}{3} (1 \pm s )^3  \right) 
Z(\sigma_{0},2 \mp s).
\end{equation}
By Theorem 4.6 of \cite{GP}, the Selberg zeta function 
 $Z(\sigma_0,s)$ has a simple zero at $s=0$, which corresponds to the zero eigenvalue of the selfadjoint Hodge Laplacian acting on $\Omega^0(\mathcal{M}_\Gamma)$.
Combining these, we have
\begin{equation}\label{e:cont-F-ruelle1}
\frac{Z(\sigma_0, -s+2)}{Z(\sigma_0, s+2)} \Big | _{s=0} =\frac{Z(\sigma_0, s)}{Z(\sigma_0, -s)} \Big|_{s=0}=-1.
\end{equation}
Again, by \eqref{e:funct-Selberg-zeta},
\begin{equation}
Z(\sigma_{-2},1+s) = e^{-i\pi \eta(D(\sigma_2))} \exp\left( \frac{\mathrm{Vol}(\mathcal{M}_\Gamma)}{\pi} \left(\frac{1}{3} s^3- s \right) \right) 
Z(\sigma_{2},1-s).
\end{equation}
noting $\eta(D(\sigma_2))=-\eta(D(\sigma_{-2}))$. Hence,
\begin{equation}\label{e:cont-F-ruelle2}
 \frac{Z(\sigma_{-2}, s+1)}{Z(\sigma_{-2}, -s+1)}\, \Big|_{s=0}= \frac{Z(\sigma_{2}, -s+1)}{Z(\sigma_{-2}, -s+1)}\, \Big|_{s=0} \, e^{-i\pi \eta(D(\sigma_2))} = 1.
\end{equation}
The last equality follows from the main theorem at p.2 of \cite{Mil}. By \eqref{e:rel-F-ruelle-rel2-1},  \eqref{e:cont-F-ruelle1}, and \eqref{e:cont-F-ruelle2},
we conclude that the equality \eqref{e:F-rel} holds.

To prove the statement for $G_0(s)$, let us introduce the following zeta functions 
\begin{equation}
Z^e(s)=Z(\sigma_{1},s) Z(\sigma_{-1},s), \qquad Z^o(s)=\frac{Z(\sigma_{1},s)}{Z(\sigma_{-1},s)}.
\end{equation}
By \eqref{e:funct-Selberg-zeta} and $\eta(D(\sigma_1))=-\eta(D(\sigma_{-1}))$, we have
\begin{equation}\label{e:p6}
\begin{split}
&Z^e(1+s)  = \exp \Big(\frac{2}{\pi}\mathrm{Vol}(\mathcal{M}_\Gamma) (\frac{s^3}{3}-\frac{s}{4}) \Big)\, Z^e(1-s),\\
&Z^o(1+s) Z^o(1-s) = \exp \big(2\pi i \eta(D(\sigma_1))\big).
\end{split}
\end{equation}
By \eqref{e:F-exp1}, we have
\begin{equation}
\begin{split}
G_0(s)^2=& \frac{Z(\sigma_{-1}, s+\frac12)^2}{Z(\sigma_{1}, s+\frac 32)^2}\\
=& \Big(\frac{Z(\sigma_{-1}, s+\frac12)}{Z(\sigma_{1}, s+\frac12)}\cdot \frac{Z(\sigma_{-1}, s+\frac 32)}{Z(\sigma_{1}, s+\frac 32)}\Big) \cdot
\Big(\frac{Z(\sigma_{1}, s+\frac12) Z(\sigma_{-1}, s+\frac12)}{Z(\sigma_{1}, s+\frac32) Z(\sigma_{-1}, s+\frac32)} \Big) \\
=& Z^o(s+\frac12)^{-1}\, Z^o(s+\frac32)^{-1}\, Z^e(s+\frac12)\, Z^e(s+\frac32)^{-1}\\
=& Z^o(s+\frac12)^{-1}\, Z^o(s+\frac32)^{-1}\, \frac{Z^e(s+\frac12)}{Z^e(-s+\frac12)}\, \frac{Z^e(-s+\frac12)}{Z^e(s+\frac32)}.
\end{split}
\end{equation}
Using this and the equalities in \eqref{e:p6}, we have
\begin{equation}\label{e:p7}
G_0(s)^2= \frac{Z^e(s+\frac12)}{Z^e(-s+\frac12)}\, \exp\Big(\, -\frac{2}{\pi}\mathrm{Vol}(\mathcal{M}_\Gamma) \big( \frac13(s+\frac12)^3-\frac14(s+\frac12)\big)\, \Big)
\, \exp(-2\pi i \eta(D(\sigma_1))).
\end{equation}
Hence, the function $G_0(s)$ is regular at $s=0$. As before, to remove the ambiguity from the first part at $s=0$ on the right hand side of \eqref{e:p7}, we take
its square and obtain the expression of $G_0(s)^4$ at $s=0$ given in \eqref{e:G-rel}. 
\end{proof}


\section{Determinant and Torsion}

\subsection{Determinant and Selberg zeta function}
We start with

\begin{lemma}\label{l:asymp} For $p=0,1$, we have the following asymptotics as $t\to0$,
\begin{equation}\label{e:heat-asymp}
\mathrm{Tr} \left( e^{-t\Delta^\flat_p} \right) \sim  \frac{1}{4\pi^2}  \mathrm{dim} (V_\rho)\, \mathrm{Vol}(\mathcal{M}_\Gamma)\, \left( a_{1}(p) t^{-3/2} + a_{2}(p) t^{-1/2} \right) + O(t^{1/2}) 
\end{equation}
where $a_1(0)= \frac12\sqrt{\pi}$, $a_2(0)=-\frac12\sqrt{\pi}$, $a_1(1)= \frac32 \sqrt{\pi}$, $a_2(1)= \frac32\sqrt{\pi}$. 
\end{lemma}

\begin{proof}
The small time asymptotics of $\mathrm{Tr}(e^{-t(\Delta^\flat_0-1)})$, $\mathrm{Tr}(e^{-\Delta^\flat_1}) -\mathrm{Tr}(e^{-t\Delta^\flat_0})$ 
follow from the ones of the right hand sides
of \eqref{e:heat-trace0} and \eqref{e:heat-trace} respectively. The second part of the right hand side from hyperbolic element $\gamma$'s has the size of $O(e^{-c/t})$ for a constant $c>0$,
hence the main contribution is given by the first part from the identity element in $\Gamma$. Hence, we have the following asymptotics
\begin{equation}\label{e:asymp-rel}
\begin{split}
&\mathrm{Tr}(e^{-t(\Delta^\flat_0-1)})  \sim  \frac{1}{4\pi^2}  \mathrm{dim} (V_\rho)\, \mathrm{Vol}(\mathcal{M}_\Gamma)\, \left( a_{1}(0) t^{-3/2} + a_{2}(0) t^{-1/2} \right) + O(t^{1/2}),\\
&\mathrm{Tr}(e^{-\Delta^\flat_1}) -\mathrm{Tr}(e^{-t\Delta^\flat_0}) \sim  \frac{1}{4\pi^2}  \mathrm{dim} (V_\rho)\, \mathrm{Vol}(\mathcal{M}_\Gamma)\, \left( a_{1}(1) t^{-3/2} + a_{2}(1) t^{-1/2} \right) + O(t^{1/2})
\end{split}
\end{equation}
where   $a_1(0)= \frac12\sqrt{\pi}$, $a_2(0)=0$, $a_1(1)=  \sqrt{\pi}$, $a_2(1)= 2\sqrt{\pi}$. From these, the equality \eqref{e:heat-asymp}
follows easily.
\end{proof}

Now we choose a complex number $s$ is in $\Lambda_\epsilon\setminus B_r$ for some $\epsilon>0$ and $r>0$ such that
the spectrum of $\Delta^\flat_p+s^2$ for $p=0,1$ lies on the right half plane with its real part is bigger than some $\delta>0$.
Under this assumption for a fixed $s$, we consider
\begin{equation}\label{e:zeta-det}
\zeta_p(z,s) =\frac{1}{\Gamma(s)}\int^\infty_0 t^{z-1} e^{-ts^2} \mathrm{Tr} \left( e^{-t\Delta_p^\flat} \right)\, dt.
\end{equation}
Note that this integral converges absolutely and uniformly on compact subsets of the half plane $\mathrm{Re}(z) >\frac32$.
Using Lemma \ref{l:asymp}, one can show that the $\zeta_p(z,s)$ has a meromorphic extension with respect to $z$ over $\mathbb{C}$, and it is regular at $z=0$.  For $p=0,1$, we define
\begin{equation}\label{e:def-det} 
\mathrm{det}\left( \Delta_p^\flat+s^2 \right)= \exp \left( -\frac{d}{dz}\Big|_{z=0} \, \zeta_p(z,s) \right) .
\end{equation}
In a similar way, we define $\mathrm{det}\left(\Delta^\flat_0-1+s^2\right)$.

\begin{lemma}\label{l:asymp-large-s}
As $s\to\infty$ in the region $\Lambda_\epsilon\setminus B_r$,
\begin{equation}\label{e:det-asymp-infty}
\begin{split}
&\log \mathrm{det} \left(\Delta_0^\flat-1+s^2\right) =  \frac{1}{2\pi}  {\mathrm{dim}(V_\rho)\,\mathrm{Vol}(\mathcal{M}_\Gamma)} \left( -\frac13 s^3 \right)  +O(s^{-1}),\\
&\log \mathrm{det} \left(\Delta_1^\flat+s^2\right)- \log \mathrm{det} \left(\Delta_0^\flat+s^2\right)= \frac{1}{\pi}  {\mathrm{dim}(V_\rho)\,\mathrm{Vol}(\mathcal{M}_\Gamma)} \left( s -\frac13 s^3 \right) +O(s^{-1}).
\end{split}
\end{equation}
Here we take the principal branch for the logarithm. 
\end{lemma} 

\begin{proof} The asymptotics as $s\to\infty$ follows from \eqref{e:asymp-rel} and 
$\int^\infty_0 t^{z-1} e^{-ts^2}  \, dt = s^{-2z}\, \Gamma(z)$.
\end{proof}

By  Lemma \ref{l:asymp}, we have 
\begin{equation}\label{e:det-zeta-rel}
\begin{split}
&\frac{1}{2s} \frac{d}{ds} \log \mathrm{det} \left(\Delta^\flat_p+s^2\right) - \frac1{2s_0} \frac{d}{ds}\Big{|}_{s=s_0} \log \mathrm{det} 
\left(\Delta^\flat_p+s^2\right)\\
=&\lim_{z\to 0} \left( -\frac{1}{2s} \frac{d}{ds}\left(\Gamma(z) \zeta_p(z,s)\right) +\frac{1}{2s_0} \frac{d}{ds}\Big{|}_{s=s_0}\left(\Gamma(z)\zeta_{p}(z,s)\right) \right)\\
=&\, \int_0^\infty \left( e^{-ts^2} - e^{-ts_0^2} \right) \mathrm{Tr}( e^{-t\Delta_p^\flat})\, dt.
\end{split}
\end{equation}

Now, we deal with the geometric side of \eqref{e:heat-trace} as we did for the spectral side, that is, we multiply $e^{-ts^2}$ to the geometric side of \eqref{e:heat-trace} and take integral $\int^\infty_0 \cdot\, dt$. First, to deal with the terms from hyperbolic elements,
from \eqref{e:def-D-gamma} we recall
\begin{equation}\label{e:D-expan}
\begin{split}
D(\gamma)^{-1}= &\ e^{-\ell_\gamma} (1-e^{-(\ell_\gamma+i\theta_\gamma)})^{-1} (  1-e^{-(\ell_\gamma-i\theta_\gamma)})^{-1}\\
                           = &\ e^{-\ell_\gamma} \prod_{p\geq 0, q\geq 0} e^{-p(l_\gamma+i\theta_\gamma)} e^{-q(l_\gamma-i\theta_\gamma)}.
\end{split}
\end{equation}
Then we obtain the following equalities
\begin{equation}\label{e:selberg-zeta-test}
\begin{split}
& \sum_{[\gamma]\neq [e]} \frac{\ell_{\gamma_0}\,  \mathrm{tr}\,\rho(\gamma)}{ D(\gamma)} 
e^{i\theta_\gamma} \int^\infty_{0} \frac{1}{\sqrt{4\pi t}} e^{-\ell_\gamma^2/4t-ts^2} \, dt 
= \frac{1}{2s} \sum_{[\gamma]\neq [e]} \frac{\ell_{\gamma_0}\,  \mathrm{tr}\,\rho(\gamma)}{ D(\gamma)} 
e^{i\theta_\gamma} e^{-s\ell_\gamma}\\
=&\frac{1}{2s}\sum_{[\gamma]\neq [e]} {\ell_{\gamma_0}\,  \mathrm{tr}\,\rho(\gamma)}
e^{i\theta_\gamma}  \prod_{p\geq 0, q\geq 0} e^{-p(l_\gamma+i\theta_\gamma)} e^{-q(l_\gamma-i\theta_\gamma)} e^{-(s+1)\ell_\gamma}\\
=&\frac1{2s}\sum_{[\gamma]_{\mathrm{p}}\neq [e]} \sum_{m=1}^\infty \frac{\ell_{\gamma}}{m}\, \left( \mathrm{tr}\,\rho(\gamma)
e^{i\theta_\gamma}  \prod_{p\geq 0, q\geq 0} e^{-p(l_\gamma+i\theta_\gamma)} e^{-q(l_\gamma-i\theta_\gamma)} e^{-(s+1)\ell_\gamma}\right)^m\\
=&-\frac1{2s} \sum_{[\gamma]_{\mathrm{p}}\neq [e]}\, \ell_\gamma \, \log\, \mathrm{det}\left(\mathrm{Id} - \rho(\gamma)
e^{i\theta_\gamma}  \prod_{p\geq 0, q\geq 0} e^{-p(l_\gamma+i\theta_\gamma)} e^{-q(l_\gamma-i\theta_\gamma)} e^{-(s+1)\ell_\gamma}\right)\\
=& \frac1{2s} \frac{d}{ds}\log Z_\rho(\sigma_2, s+1).
\end{split}
\end{equation}
Here the sums at the third and forth lines run over the set of conjugacy classes in $\Gamma$ of \emph{primitive} hyperbolic elements in $\Gamma$.
For the first equality above, we used
\begin{equation}
\int^\infty_0 \frac{1}{\sqrt{4\pi t}} e^{-\ell^2/4t-ts^2} \, dt = \frac{1}{2s} e^{-\ell s}.
\end{equation}

Repeating the same procedure for the identity contribution of the geometric side of \eqref{e:heat-trace}, 
\begin{equation}\label{e:ide-cont-test}
\begin{split}
 &\frac{1}{4\pi^2}  {\mathrm{dim}(V_\rho)\,\mathrm{Vol}(\mathcal{M}_\Gamma)}
\int_0^\infty  e^{-t s^2} \int^\infty_0\left(\lambda^2+ 1\right)\, e^{-t\lambda^2} \, d\lambda \, dt \\
=& \frac{1}{4\pi^2}  {\mathrm{dim}(V_\rho)\,\mathrm{Vol}(\mathcal{M}_\Gamma)}
\int_{0}^\infty   e^{-ts^2} \sqrt{\pi}\left(  \frac{1}2 t^{-3/2} + t^{-1/2} \right)  \, dt \\
=&\frac1{2s}  \frac{1}{2\pi}  {\mathrm{dim}(V_\rho)\,\mathrm{Vol}(\mathcal{M}_\Gamma)} \left( -s^2+1 \right). 
\end{split}
\end{equation}

Combining  \eqref{e:heat-trace}, \eqref{e:det-zeta-rel}, \eqref{e:selberg-zeta-test}, and \eqref{e:ide-cont-test},  
\begin{equation}\label{e:selberg-det}
\begin{split}
&\ \frac1s \frac{d}{ds} \log \frac{ \mathrm{det} \left(\Delta^\flat_1+s^2\right)}{\mathrm{det} \left( \Delta^\flat_0+s^2\right)}
 - \frac1{s_0} \frac{d}{ds}\Big|_{s=s_0} \log \frac{ \mathrm{det} \left(\Delta^\flat_1+s^2\right)}{\mathrm{det} \left( \Delta^\flat_0+s^2\right)}\\
=&\ \frac1s \frac{d}{ds}\log \left(Z_\rho(\sigma_2, s+1)Z_\rho(\sigma_{-2}, s+1)\right) \\
&-  \frac1{s_0} \frac{d}{ds}\Big|_{s=s_0}\, \log \left(Z_\rho(\sigma_2, s+1)Z_\rho(\sigma_{-2}, s+1)\right)\\
&+ \frac1s  \frac{1}{\pi}  {\mathrm{dim}(V_\rho)\,\mathrm{Vol}(\mathcal{M}_\Gamma)} \left( -s^2+1 \right) -\frac1{s_0}  \frac{1}{\pi}  {\mathrm{dim}(V_\rho)\,\mathrm{Vol}(\mathcal{M}_\Gamma)} \left( -s_0^2+1 \right),
\end{split}
\end{equation}
which implies
\begin{equation}\label{e:det-zeta2}
\begin{split}
& \log \mathrm{det}\left(\Delta^\flat_1+s^2\right)-\log {\mathrm{det} \left( \Delta^\flat_0+s^2\right)}\\
= &\, \log Z_\rho(\sigma_2, s+1)+  \log Z_\rho(\sigma_{-2}, s+1)\\
&+ \frac{1}{\pi}  {\mathrm{dim}(V_\rho)\,\mathrm{Vol}(\mathcal{M}_\Gamma)} \left( s -\frac13 s^3 \right) + c_2s^2 +c_0
\end{split}
\end{equation}
for some constants $c_2,c_0$.
In a similar way using \eqref{e:heat-trace0}, we obtain
\begin{equation}\label{e:det-zeta0}
\begin{split}
&\log {\mathrm{det} \left( \Delta^\flat_0-1+s^2\right)}\\
=& \log Z_\rho(\sigma_0, s+1)+ \frac{1}{2\pi}  {\mathrm{dim}(V_\rho)\,\mathrm{Vol}(\mathcal{M}_\Gamma)} \left( -\frac13 s^3 \right) + d_2s^2 +d_0
\end{split}
\end{equation}
for some constants $d_2, d_0$. By Lemma \ref{l:asymp-large-s} and the fact that $\log Z_\rho(\sigma_k,s)$ decays exponentially as $s\to\infty$, the constants $c_2,c_0,d_2,d_0$ are trivial.  Hence, we have

\begin{proposition} The following equalities hold for $s\in\mathbb{C}$,
\begin{align}\label{e:det-exp0}
\mathrm{det} \left( \Delta^\flat_0-1+s^2\right)= Z_\rho(\sigma_0, s+1) \exp\left(-\frac{1}{6\pi} {\mathrm{dim}(V_\rho)\,\mathrm{Vol}(\mathcal{M}_\Gamma)} s^3\right),
\end{align}
\begin{align}\label{e:det-exp1}
\frac{\mathrm{det}\left(\Delta^\flat_1+s^2\right)}{ {\mathrm{det} \left( \Delta^\flat_0+s^2\right)}} = Z_\rho(\sigma_2, s+1)Z_\rho(\sigma_{-2}, s+1)
\exp\left( \frac{1}{\pi}  {\mathrm{dim}(V_\rho)\,\mathrm{Vol}(\mathcal{M}_\Gamma)} \left( s -\frac13 s^3 \right) \right). 
\end{align}
\end{proposition}
\begin{proof}
Although the left hand sides of \eqref{e:det-exp0} and \eqref{e:det-exp1} are defined a priori for
$s$ with some conditions, these can be extended over $\mathbb{C}$ by the meromorphicity of $Z_{\rho}(\sigma_k,s)$.
\end{proof}

By \eqref{e:ruelle-decomp1}, \eqref{e:det-exp0}, and \eqref{e:det-exp1},
\begin{equation}\label{e:ruelle-torsion}
\begin{split}
R_{\rho}(s)=&\frac{ Z_{\rho}(\sigma_0, s) \, Z_{\rho}(\sigma_0,s+2)}{Z_{\rho}(\sigma_2,s+1)\, Z_{\rho}(\sigma_{-2},s+1)}\\
=&\frac{\mathrm{det}\left(\Delta^\flat_0-1+(s-1)^2\right)\, \mathrm{det}\left( \Delta_0^\flat-1 +(s+1)^2\right) {\mathrm{det} \left( \Delta^\flat_0+s^2\right)} }{\mathrm{det}\left(\Delta^\flat_1+s^2\right)}\\
&\cdot \exp\left( \frac{2s}{\pi} {\mathrm{dim}(V_\rho)\,\mathrm{Vol}(\mathcal{M}_\Gamma)} \right).
\end{split}
\end{equation}
This equality implies the following functional equation of $R_{\rho}(s)$,
\begin{theorem}\label{t:funct-eq}
The following equality holds for $s\in\mathbb{C}$,
\begin{equation}\label{e:ruelle-funct}
R_{\rho}(s)= R_{\rho}(-s) \exp \left( \frac{4s}{\pi} {\mathrm{dim}(V_\rho)\,\mathrm{Vol}(\mathcal{M}_\Gamma)} \right).
\end{equation}
\end{theorem}
Noting \eqref{e:comp-volume}, the equation \eqref{e:ruelle-funct} is compatible with Theorem 1.1 of \cite{GP} which holds for unitary representation $\rho$.
Let us remark that the Ruelle zeta function in this paper is the inverse of the one of \cite{GP}.

\subsection{Reidemeister torsion} For an $n$-dimensional vector space over $\mathbb{C}$, let $v=(v_1,\ldots, v_n)$ and $w=(w_1,\ldots, w_n)$ are two bases for $V$. Let $[w/v]$ denote the determinant of the matrix $T=(t_{ij})$ representing
the change of base from $v$ to $w$, that is, $w_i=\sum t_{ij} v_j$. Suppose 
\begin{equation} 
C: C_N\ \stackrel{\partial}{\rightarrow}\ C_{N-1}\ \stackrel{\partial}{\rightarrow} \ \cdots \ \stackrel{\partial}{\rightarrow}\  C_1\ 
\stackrel{\partial}{\rightarrow}\ C_0 
\end{equation} 
is a
chain complex of finite complex modules. Let $Z_q$ denote the kernel of $\partial$ in $C_q$, $B_q\subset C_q$ the image of $C_{q+1}$ under $\partial$, and $H_q(C)=Z_q/B_q$ the $q$-th homology group of $C$. Choose a base $b_q$ for
$B_q$ for each $q$, and let $\tilde{b}_{q-1}$ be an independent set in $C_q$ such that $\partial \tilde{b}_{q-1}= b_{q-1}$, and $\tilde{h}_q$ an independent set in $Z_q$ representing a base $h_q$ of $H_q(C)$. Then $(b_q,
\tilde{h}_q, \tilde{b}_{q-1})$ is a base for $C_q$. For a given preferred base $c_q$ for $C_q$, note that $[b_q,\tilde{h}_q, \tilde{b}_{q-1}/c_q]$ depends only on $b_q$, $h_q$, $b_{q-1}$, hence we denote it by $[b_q,{h}_q,
{b}_{q-1}/c_q]$. The \emph{torsion} $\tau(C)$ of the chain complex $C$ is the nonzero complex number defined by 
\begin{equation} \label{e:def-alg-torsion}
\mathcal{T}(C)= \prod_{q=0}^N [b_q,{h}_q, {b}_{q-1}/c_q]^{(-1)^q}. 
\end{equation} 
Note that $\mathcal{T}(C)$
depends only on the choice of the bases $c_q$, $h_q$, but not on the choice of the bases $b_q$.

Let $K$ be a finite cell complex and $\tilde{K}$ the simply connected covering space of $K$ with the fundamental group $\pi_1$ of $K$ acting as deck transformations on $\tilde{K}$. Regarding that $\tilde{K}$ is a just the set of
translates of a fundamental domain under $\pi_1$, the complex cochain groups $C^q(\tilde{K})$ become modules over the complex group algebra $\mathbb{C}(\pi_1)$ with a preferred base consisting of the dual element of cells of $K$. Relative to these
preferred base, the boundary operator on the right $C(\pi_1)$-module $C^q(\tilde{K})$ is a matrix with coefficients in $\mathbb{C}(\pi_1)$. For a representation $\rho$ of $\pi_1(K)$ into $\mathrm{SL}(\mathbb{C}^N)$, define the chain complex
$C(K,\rho)$ by
\begin{equation} 
C^q(K,\rho) = C^q(\tilde{K}) \otimes _{\mathbb{C}(\pi_1)} \mathbb{C}^N 
\end{equation} 
where $\mathbb{C}^N$ is considered as left $\mathbb{C}(\pi_1)$-module via the action of $\rho$. 
The boundary map of $C(K,\rho)$ is defined to be the dual map of the boundary map of the cell complex. We choose a
preferred base $e_i\otimes x_j$ where  $e_i$ runs through the preferred base of $C^q(\tilde{K})$ and $x_j$ runs through a base for $\mathbb{C}^N$. 
Now the \emph{Reidemeister torsion} $\mathcal{T}(K,\rho)$ attached to
the representation $\rho$ is defined by 
\begin{equation} \label{e:def-Reide}
\mathcal{T}(K,\rho)=\mathcal{T}(C(K,\rho)).
\end{equation}  
A different choice of the preferred bases $e_i$ can give at most sign change of $\mathcal{T}(K,\rho)$ since $\rho$ is a
representation into $SL(\mathbb{C}^N)$. A different choice of the base $x'_j$ for $\mathbb{C}^N$ can also give the change by the factor $[x'/x]^{\chi(C)}$ where $\chi(C)$ denotes the Euler characteristic of $C$. Hence, if
$\chi(C)=0$, the Reidemeister torsion $\mathcal{T}(C(K,\rho))$ is well-defined as an invariant with a value in $\mathbb{C}^*/\{\pm 1\}$ depending only on the choice of the bases $h_q$ for $H_q(C)$. By \cite{Mi}, it is known that
$\mathcal{T}(C(K,\rho))$ is a combinatorial invariant of $(K, \rho)$. Hence, if $\mathcal{M}$ is a compact oriented manifold, any smooth triangulation of $\mathcal{M}$ gives the same Reidemeister torsion. We denote it by $\mathcal{T}(\mathcal{M},\rho)$.

When the cohomology groups $H^*(\mathcal{M},\rho):= H_{*}(C(K, {\rho}))$ are trivial for a smooth triangulation $K$ of $\mathcal{M}$, 
the square of Reidemeister torsion $\mathcal{T}^2(\mathcal{M},\rho)$
is a well defined complex number. 

To state the theorem 10.1 of \cite{CM}, we need some preparation to introduce the analytic torsion defined by the non-selfadjoint operators $\Delta_p^\flat$
in the context of the subsection \ref{ss:hodge}. Hence, the operator $\Delta_p^\flat$ acts on $C^\infty(\mathcal{M}, E_\rho)$ where $\mathcal{M}$ is a closed Riemannian manifold and $E$ is a flat vector bundle defined by a representation $\rho:\pi_1(\mathcal{M})\to \mathrm{SL}(\mathbb{C}^N)$.
In general, the non-selfadjoint operator $\Delta_p^\flat$ acting on $\Omega^p(\mathcal{M}, E_\rho)$ may have a generalized eigenvalue with non-positive real part. Hence, the definitions
\eqref{e:zeta-det} and \eqref{e:def-det} do not work simply if we put $s=0$ at \eqref{e:zeta-det} and \eqref{e:def-det}. For this, we recall the construction in the section of \cite{CM}. Let $r>0$ be a real number that is not the real part of any generalized eigenvalues of $\Delta_p^\flat$.  Let $\Pi_{p,r}$ denote the spectral projection on the span of the
generalized eigenvectors with generalized eigenvalues with real part less than $r$. Noting that Lemma \ref{l:asymp-large-s} still holds for the heat operator of  $\Delta_{p,r}^\flat:=(\mathrm{Id}-\Pi_{p,r}) \Delta_p^\flat$ since $\Pi_{p,r}$ is a smoothing operator, we define
\begin{equation}\label{e:def-det1} 
\mathrm{det}\Delta_{p,r}^\flat = \exp \left( -\frac{d}{dz}\Big|_{z=0} \, \zeta_{p,r}(z) \right) 
\end{equation}
where
\begin{equation}\label{e:zeta-det1}
\zeta_{p,r}(z) =\frac{1}{\Gamma(s)}\int^\infty_0 t^{z-1}  \mathrm{Tr} \left( e^{-t\Delta_{p,r}^\flat} \right)\, dt.
\end{equation}
Let us denote the zero generalized eigenspace of $\Delta^\flat_p$ by $\Omega^p_0(\mathcal{M}, E_{\rho})$.
The cohomology of the complex $(\Omega^*_0(\mathcal{M}, E_{\rho}),d)$ is the same as the cohomology of $(\Omega^*(\mathcal{M}, E_\rho), d)$.
We define its torsion as in \eqref{e:def-alg-torsion} by
\begin{equation}\label{e:def-zero-torsion}
\mathcal{T}_0(\mathcal{M}, \rho)= \mathcal{T}( \Omega^*_0(\mathcal{M}, E_{\rho}),d).
\end{equation}
Then the \emph{analytic torsion} $T(\mathcal{M},\rho)$ is defined by
\begin{equation}\label{e:def-anal-tor}
T(\mathcal{M},\rho) =\mathcal{T}_0(\mathcal{M}, \rho)^2 \cdot 
\prod_{p=1}^N \left(\mathrm{det}\Delta_{p,r}^\flat \right)^{p (-1)^{p+1}}\cdot \prod_{p=1}^N \left(\prod_{\lambda_{p,j}\in S(p,r)} \lambda_{p,j}\right)^{p (-1)^{p+1}}. 
\end{equation}
Here $S(p,r)$ denotes the set of
all the nonzero generalized eigenvalues counted with multiplicities with real part less than $r$. By the theorem 8.3 of \cite{CM}, the above definition of $T(\mathcal{M},\rho)$ in \eqref{e:def-anal-tor} does not depend on the choice of $r$.
When the cohomology groups $H^*(\mathcal{M},\rho)$ are trivial, the following equality holds between two complex valued
invariants
\begin{equation}\label{e:Cap-Mil}
\mathcal{T}^2(\mathcal{M},\rho) = T(\mathcal{M}, \rho)
\end{equation}
by the theorem 10.1 of \cite{CM}.

\section{Proof of Main theorem}

By Remark \ref{r:reg-Ruelle-zero},  $R_{\rho_m}(s)$ is regular and has a finite value at $s=0$ for $m\geq 3$. By this fact and  \eqref{e:ruelle-torsion},
we have
\begin{equation}\label{e:ruelle-torsion-0} 
R_{\rho_m}(0) =\lim_{s\to 0} \
\frac{\mathrm{det}\left(\Delta^\flat_0+s^2-2s\right)\, \mathrm{det}\left( \Delta_0^\flat+s^2+2s\right) {\mathrm{det} \left( \Delta^\flat_0+s^2\right)} }{\mathrm{det}\left(\Delta^\flat_1+s^2\right)}.
\end{equation}
For $p=0,1$, we take a sufficiently small $r>0$ such that the real parts of the generalized eigenvalues of $\Delta_p^\flat$ with positive real parts are bigger than $r$. Then, it is easy to check that
\begin{equation}\label{e:limit-zeta-det}
\begin{split}
\lim_{s\to 0}\, s^{-2h_p}\mathrm{det} \left(\Delta_p^\flat+s^2\right) =& \lim_{s\to 0} \mathrm{det}  \left(\Delta_{p,r}^\flat+s^2\right)\cdot s^{-2h_p} \prod_{\lambda_{p,j}\in S(p,r)} \left(\lambda_{p,j}+s^2 \right)\\
=&\  \mathrm{det}  \Delta_{p,r}^\flat \cdot \prod_{\lambda_{p,j}\in S(p,r)} \lambda_{p,j} 
\end{split}
\end{equation}
where $h_p$ denotes the multiplicities of the zero generalized eigenvalue of $\Delta^\flat_p$.
Similar equalities hold for other factors $\mathrm{det}\left( \Delta_0^\flat+s^2\pm2s\right)$ on the right hand side of \eqref{e:ruelle-torsion-0}. 
 By \eqref{e:ruelle-torsion-0}, we have $2h_0=h_1$. Using this and 
by \eqref{e:def-anal-tor}, \eqref{e:ruelle-torsion-0}, and \eqref{e:limit-zeta-det},
\begin{equation}\label{e:limit-det-dec}
\begin{split}
R_{\rho_m}(0)=\, & \lim_{s\to 0} \
\frac{s^{-4h_0}\,  \mathrm{det}\left(\Delta^\flat_0+s^2-2s\right)\, \mathrm{det}\left( \Delta_0^\flat+s^2+2s\right) {\mathrm{det} \left( \Delta^\flat_0+s^2\right)} }
{s^{-2h_1}\,\mathrm{det}\left(\Delta^\flat_1+s^2\right)}\\
=\, &\frac{ \left(\mathrm{det} \Delta^\flat_{0,r} \cdot  \prod_{\lambda_{0,j}\in S(0,r)} \lambda_{0,j} \right)^3}
{\mathrm{det} \Delta^\flat_{1,r} \cdot  \prod_{\lambda_{1,j}\in S(1,r)} \lambda_{1,j} }\\
=\, &\frac{ \left(\mathrm{det} \Delta^\flat_{1,r} \cdot  \prod_{\lambda_{1,j}\in S(1,r)} \lambda_{1,j} \right)\, \left(\mathrm{det} \Delta^\flat_{3,r} \cdot  \prod_{\lambda_{3,j}\in S(3,r)} \lambda_{3,j} \right)^3}
{\left(\mathrm{det} \Delta^\flat_{2,r} \cdot  \prod_{\lambda_{2,j}\in S(2,r)} \lambda_{2,j}\right)^2 }\\
=\, & \frac{T(\mathcal{M}_\Gamma, \rho_m)}{\mathcal{T}_0(\mathcal{M}_\Gamma, \rho_m)^2}.
\end{split}
\end{equation}
For the third equality of \eqref{e:limit-det-dec}, we used 
\begin{equation}\label{e:Delta-star}
\left(\star\otimes\mathrm{Id}_{E_{\rho_m}}\right) \, \Delta_p^\flat = \Delta_{3-p}^\flat\, \left(\star\otimes\mathrm{Id}_{E_{\rho_m}}\right) \qquad \text{over} \quad \Omega^p(\mathcal{M}_\Gamma, E_{\rho_m}),
\end{equation}
which follows from  \eqref{e:d-star-def-ext}.

For the representation $\rho_m:\Gamma\to \mathrm{SL}(S^m(\mathbb{C}^2))\subset \mathrm{SL}(\mathbb{C}^{m+1})$, the cohomology groups $H^{*}(\mathcal{M}_\Gamma, E_{\rho_m})$ are trivial by the theorem 6.7 at p.226 of \cite{BW80}. Hence,
the square of the Reidemeister torsion $\mathcal{T}^2(\mathcal{M}_\Gamma,\rho_m)$ is a well-defined complex valued invariant.
By \eqref{e:Cap-Mil}, we have $\mathcal{T}^2(\mathcal{M}_\Gamma,\rho_m)=T(\mathcal{M}_\Gamma,\rho_m)$. 
Combining this and \eqref{e:limit-det-dec},
\begin{equation}\label{e:ruelle-reide-tor}
 \frac{\mathcal{T}(\mathcal{M}_\Gamma,\rho_m)^2}{\mathcal{T}_0(\mathcal{M}_\Gamma, \rho_m)^2}= R_{\rho_m}(0).
\end{equation}


From now on we split the proof when $m$ is even or odd.

When $m=2(n-1)$ with $n\geq 3$, by  \eqref{e:F-rel-prim} and \eqref{e:ruelle-reide-tor},
\begin{equation}\label{e:main-proof-0}
\begin{split}
\left( \frac{\mathcal{T}_0(\mathcal{M}_\Gamma, \rho_{2(n-1)})} {\mathcal{T}(\mathcal{M}_\Gamma, \rho_{2(n-1)})} \right)^{4} = &\, \exp\left(2i\pi(\eta(D(\sigma_{2n}))-\eta(D(\sigma_{2(n-1)})))\right)\\
&\qquad\qquad  \, \cdot \exp\left( \frac{2}{\pi}\, \left(2n^2-2n+\frac13\right) \mathrm{Vol}(\mathcal{M}_\Gamma) \right) \cdot F_n^{4}.
\end{split}
\end{equation}
Here
\begin{equation}\label{e:def-F-final}
F_n:=F_n(0) =\prod_{k=n}^\infty R\left(\sigma_{-2k}, k\right)= \prod_{k=n}^\infty \prod_{[\gamma]} \left(1- e^{-k(\ell_\gamma+i\theta_\gamma)}\right).
\end{equation}
Recall that $s=0$ lies in the convergence domain of $F_n(s)$ for $n\geq 3$ (see \eqref{e:def-Fs}).
By  \eqref{e:APS} and \eqref{e:main-proof-0}, 
\begin{equation}\label{e:main-proof-F}
\begin{split}
\left(\frac{\mathcal{T}_0(\mathcal{M}_\Gamma, \rho_{2(n-1)})}{\mathcal{T}(\mathcal{M}_\Gamma, \rho_{2(n-1)})}\right)^{12}=&\, \exp\left( 6\pi i \theta_{2(n-1)} \right)\\
& \, \exp\left( \frac{2}{\pi} \left(6{n^2}-6n+1\right) \left({\mathrm{Vol}(\mathcal{M}_\Gamma)}+i 2\pi^2 \mathrm{CS}(\mathcal{M}_\Gamma)\right)\, \right) \cdot F_n^{12}
\end{split}
\end{equation}
where
\begin{equation}\label{e:cn-eta}
\theta_{2(n-1)}= \eta(D(\sigma_{2n}))-\eta(D(\sigma_{2(n-1)}))-\left(6n^2-6n+{1}\right) \eta(D(\sigma_2)) .
\end{equation}
Let us remark that $\theta_0=0$ with $n=1$ by the definition as expected from \eqref{e:F-rel}. Hence, $\theta_{2(n-1)}$ can be understood as an anomaly
for nonzero $m=2(n-1)$.
This completes the proof of \eqref{e:main-theorem}.

When $m=2n-1$ with $n\geq 2$, by \eqref{e:G-rel-prim} and \eqref{e:ruelle-reide-tor},
\begin{equation}\label{e:main-proof-f5}
\begin{split}
\left(\frac{\mathcal{T}_0(\mathcal{M}_\Gamma, \rho_{2n-1})}{\mathcal{T}(\mathcal{M}_\Gamma, {\rho_{2n-1}})}\right)^{4} = &\, \exp\left(2i\pi(\eta(D(\sigma_{2n+1}))-\eta(D(\sigma_{2n-1})))\right)\\
&\qquad\qquad  \, \exp\left( \frac{2}{\pi}\, \left(2n^2-\frac16\right) \mathrm{Vol}(\mathcal{M}_\Gamma)\right) \cdot G_n^{4}.
\end{split}
\end{equation}
Here
\begin{equation}
G_n:=G_n(0) =\prod_{k=n}^\infty R\left(\sigma_{-(2k+1)}, k+\frac12\right)= \prod_{k=n}^\infty \prod_{[\gamma]} \left(1- e^{-(k+1/2)(\ell_\gamma+i\theta_\gamma)}\right).
\end{equation}
Recall that $s=0$ lies in the convergence domain of $G_n(s)$ for $n\geq 2$.
Rewriting \eqref{e:main-proof-f5} motivated by \eqref{e:G-rel},

\begin{theorem} \label{t:odd-case} For $m=2n-1$ with $n\geq 2$, the following equality holds
\begin{equation}\label{e:main-odd-G}
\begin{split}
\left(\frac{\mathcal{T}_0(\mathcal{M}_\Gamma, \rho_{2n-1})}{\mathcal{T}(\mathcal{M}_\Gamma, {\rho_{2n-1}})}\right)^{4} =&\, \exp\left( 2i\pi  \theta_{2n-1} \right)\\
& \, \cdot \exp\left( \frac{2}{\pi} \left(2{n^2}-\frac16\right) \left({\mathrm{Vol}(\mathcal{M}_\Gamma)}+i 3\pi^2 \eta(D(\sigma_1))\right)\, \right) \cdot G_n^{4}
\end{split}
\end{equation}
where $\theta_{2n-1}:=\eta(D(\sigma_{2n+1}))-\eta(D(\sigma_{2n-1}))-  \left(6n^2-\frac12\right) \eta(D(\sigma_1))$.
\end{theorem}

Let us remark that the equality \eqref{e:main-odd-G} is compatible with \eqref{e:G-rel} if we put $n=0$ formally and note $\eta(D(\sigma_1))=-\eta(D(\sigma_{-1}))$ (see Remark \ref{r:FG}).


\bibliographystyle{plain}

\end{document}